\documentclass[a4paper]{article}

\usepackage{a4wide}
\usepackage{amsmath, amssymb, amsthm}
\usepackage{thmtools}
\usepackage{microtype,hyperref}
\usepackage{tikz}
\usetikzlibrary{arrows}
\usetikzlibrary{shapes}
\usetikzlibrary{calc}

\usepackage{subcaption}
\usepackage{authblk}

\declaretheorem[style=plain,numberwithin=section]{theorem}
\declaretheorem[style=plain,sibling=theorem]{lemma}

\declaretheorem[style=definition,sibling=theorem]{definition}
\declaretheorem[style=definition,sibling=theorem]{example}
\declaretheorem[style=definition,sibling=theorem]{remark}
\declaretheorem[style=definition,numbered=no]{construction}

\newcommand{\Z}{\ensuremath{\mathbb{Z}}}

\newcommand{\one}{\ensuremath{\mathbf{1}}}
\renewcommand{\equiv}{\sim}

\DeclareMathOperator{\sgon}{sgon}
\DeclareMathOperator{\dgon}{dgon}
\DeclareMathOperator{\sdgon}{sdgon}

\DeclareMathOperator{\rank}{rank}

\DeclareMathOperator{\tw}{tw}

\tikzstyle{vertex}=[circle, draw, fill=white, inner sep=0pt, minimum width=5pt]
\tikzstyle{added}=[circle, draw, fill=white, inner sep=0pt, minimum width=3pt]
\tikzstyle{edge} = [] 
\tikzstyle{fatedge} = [ultra thick]

\newcommand{\graafr}[2]{
	\node[vertex] (a) at ({.5+#1},{0+#2}) {};
	\node[vertex] (b) at ({-.25+#1},{0.43+#2}) {};
	\node[vertex] (c) at ({-.25+#1},{-0.43+#2}) {};
	\node[vertex] (d) at ({1.5+#1},{0+#2}) {};
	\node[vertex] (e) at ({-.75+#1},{1.30+#2}) {};
	\node[vertex] (f) at ({-.75+#1},{-1.30+#2}) {};
	\draw[edge] (a) -- (b) -- (c) -- (a);
	\draw[edge] (a) -- (d);
	\draw[edge] (a) to [relative, out=50,in=130] (d);
	\draw[edge] (d) to [relative, out=50,in=130] (a);
	\draw[edge] (b) -- (e);
	\draw[edge] (b) to [relative, out=50,in=130] (e);
	\draw[edge] (e) to [relative, out=50,in=130] (b);
	\draw[edge] (c) -- (f);
	\draw[edge] (c) to [relative, out=50,in=130] (f);
	\draw[edge] (f) to [relative, out=50,in=130] (c);
}

\newcommand{\makeg}[4]{

    \node[vertex, color = white] (center) at ({#1},{#2}) {};
	\node[vertex] (a) at ($(center) + ({#3}:{#4})$) {};
	    \node[vertex] (a1) at ($(center)+({#3}:{#4})+(30:{#4})$) {};
	    \node[vertex] (a2) at ($(center)+({#3}:{#4})+(-30:{#4})$) {};
	\node[vertex] (b) at ($(center)+({#3+120}:{#4})$) {};
	    \node[vertex] (b1) at ($(center)+({#3+120}:{#4}) + (90:{#4})$) {};
	    \node[vertex] (b2) at ($(center)+({#3+120}:{#4}) + (150:{#4})$) {};
	\node[vertex] (c) at ($(center)+({#3+240}:{#4})$) {};
	    \node[vertex] (c1) at ($(center)+({#3+240}:{#4}) + (-90:{#4})$) {};
	    \node[vertex] (c2) at ($(center)+({#3+240}:{#4}) + (-150:{#4})$) {};
	
	\draw[edge] (a) -- (a1) -- (a2) -- (a);
	\draw[edge] (b) -- (b1) -- (b2) -- (b);
	\draw[edge] (c) -- (c1) -- (c2) -- (c);
	\draw[edge] (a) -- (b) -- (c) -- (a);
}

\newcommand{\makesmallgul}[4]{
    \node[added, color = white] (center) at ({#1},{#2}) {};
	\node[added] (a) at ($(center) + ({#3}:{#4})$) {};
	    \node[added] (a1) at ($(center)+({#3}:{#4})+(30:{#4})$) {};
	    \node[added] (a2) at ($(center)+({#3}:{#4})+(-30:{#4})$) {};
	\node[added] (b) at ($(center)+({#3+120}:{#4})$) {};
	    \node[added] (b1) at ($(center)+({#3+120}:{#4}) + (90:{#4})$) {};
	    \node[added] (b2) at ($(center)+({#3+120}:{#4}) + (150:{#4})$) {};
	\node[added] (c) at ($(center)+({#3+240}:{#4})$) {};
	    \node[added] (c2) at ($(center)+({#3+240}:{#4}) + (-150:{#4})$) {};
	
	\draw[edge] (a) -- (a1) -- (a2) -- (a);
	\draw[edge] (b) -- (b1) -- (b2) -- (b);
	\draw[edge] (c) -- (s2) -- (c2) -- (c);
	\draw[edge] (a) -- (b) -- (c) -- (a);
}

\newcommand{\makesmallgrul}[4]{
    \node[added, color = white] (center) at ({#1},{#2}) {};
	\node[added] (a) at ($(center) + ({#3}:{#4})$) {};
	    \node[added] (a1) at ($(center)+({#3}:{#4})+(30:{#4})$) {};
	    \node[added] (a2) at ($(center)+({#3}:{#4})+(-30:{#4})$) {};
	\node[added] (b) at ($(center)+({#3+120}:{#4})$) {};
	    \node[added] (b1) at ($(center)+({#3+120}:{#4}) + (90:{#4})$) {};
	    \node[added] (b2) at ($(center)+({#3+120}:{#4}) + (150:{#4})$) {};
	\node[added] (c) at ($(center)+({#3+240}:{#4})$) {};
	    \node[added] (c2) at ($(center)+({#3+240}:{#4}) + (-150:{#4})$) {};
	
	\draw[edge] (a) -- (a1) -- (a2) -- (a);
	\draw[edge] (b) -- (b1) -- (b2) -- (b);
	\draw[edge] (c) -- (s2) -- (c2) -- (c);
	\draw[edge] (a) -- (b) -- (c) -- (a);
}

\newcommand{\makesmallgdl}[4]{
    \node[added, color = white] (center) at ({#1},{#2}) {};
	\node[added] (a) at ($(center) + ({#3}:{#4})$) {};
	    \node[added] (a2) at ($(center)+({#3}:{#4})+(-30:{#4})$) {};
	\node[added] (b) at ($(center)+({#3+120}:{#4})$) {};
	    \node[added] (b1) at ($(center)+({#3+120}:{#4}) + (90:{#4})$) {};
	    \node[added] (b2) at ($(center)+({#3+120}:{#4}) + (150:{#4})$) {};
	\node[added] (c) at ($(center)+({#3+240}:{#4})$) {};
	    \node[added] (c1) at ($(center)+({#3+240}:{#4}) + (-90:{#4})$) {};
	    \node[added] (c2) at ($(center)+({#3+240}:{#4}) + (-150:{#4})$) {};
	
	\draw[edge] (a) -- (s3) -- (a2) -- (a);
	\draw[edge] (b) -- (b1) -- (b2) -- (b);
	\draw[edge] (c) -- (c1) -- (c2) -- (c);
	\draw[edge] (a) -- (b) -- (c) -- (a);
}

\newcommand{\makesmallguldl}[4]{
    \node[added, color = white] (center) at ({#1},{#2}) {};
	\node[added] (a) at ($(center) + ({#3}:{#4})$) {};
	    \node[added] (a2) at ($(center)+({#3}:{#4})+(-30:{#4})$) {};
	\node[added] (b) at ($(center)+({#3+120}:{#4})$) {};
	    \node[added] (b1) at ($(center)+({#3+120}:{#4}) + (90:{#4})$) {};
	    \node[added] (b2) at ($(center)+({#3+120}:{#4}) + (150:{#4})$) {};
	\node[added] (c) at ($(center)+({#3+240}:{#4})$) {};
	    \node[added] (c1) at ($(center)+({#3+240}:{#4}) + (-90:{#4})$) {};
	    \node[added] (c2) at ($(center)+({#3+240}:{#4}) + (-150:{#4})$) {};
	
	\draw[edge] (a) -- (s3) -- (a2) -- (a);
	\draw[edge] (b) -- (b1) -- (b2) -- (b);
	\draw[edge] (c) -- (c1) -- (c2) -- (c);
	\draw[edge] (a) -- (b) -- (c) -- (a);
}

\newcommand{\makesmallgr}[4]{
    \node[added, color = white] (center) at ({#1},{#2}) {};
	\node[added] (a) at ($(center) + ({#3}:{#4})$) {};
	    \node[added] (a1) at ($(center)+({#3}:{#4})+(30:{#4})$) {};
	    \node[added] (a2) at ($(center)+({#3}:{#4})+(-30:{#4})$) {};
	\node[added] (b) at ($(center)+({#3+120}:{#4})$) {};
	    \node[added] (b2) at ($(center)+({#3+120}:{#4}) + (90:{#4})$) {};
	\node[added] (c) at ($(center)+({#3+240}:{#4})$) {};
	    \node[added] (c1) at ($(center)+({#3+240}:{#4}) + (-90:{#4})$) {};
	    \node[added] (c2) at ($(center)+({#3+240}:{#4}) + (-150:{#4})$) {};
	
	\draw[edge] (a) -- (a1) -- (a2) -- (a);
	\draw[edge] (b) -- (s1) -- (b2) -- (b);
	\draw[edge] (c) -- (c1) -- (c2) -- (c);
	\draw[edge] (a) -- (b) -- (c) -- (a);
}

\newcommand{\makesmallgdlr}[4]{
    \node[added, color = white] (center) at ({#1},{#2}) {};
	\node[added] (a) at ($(center) + ({#3}:{#4})$) {};
	    \node[added] (a1) at ($(center)+({#3}:{#4})+(30:{#4})$) {};
	    \node[added] (a2) at ($(center)+({#3}:{#4})+(-30:{#4})$) {};
	\node[added] (b) at ($(center)+({#3+120}:{#4})$) {};
	    \node[added] (b2) at ($(center)+({#3+120}:{#4}) + (90:{#4})$) {};
	\node[added] (c) at ($(center)+({#3+240}:{#4})$) {};
	    \node[added] (c1) at ($(center)+({#3+240}:{#4}) + (-90:{#4})$) {};
	    \node[added] (c2) at ($(center)+({#3+240}:{#4}) + (-150:{#4})$) {};
	
	\draw[edge] (a) -- (a1) -- (a2) -- (a);
	\draw[edge] (b) -- (s1) -- (b2) -- (b);
	\draw[edge] (c) -- (c1) -- (c2) -- (c);
	\draw[edge] (a) -- (b) -- (c) -- (a);
}

\newcommand{\makegtwo}[4]{
    \node[vertex, color = white] (center) at ({#1},{#2}) {};
	\node[vertex] (s1) at ($(center) + ({#3}:{#4})$) {};
	\node[vertex] (s2) at ($(center)+({#3+120}:{#4})$) {};
	\node[vertex] (s3) at ($(center)+({#3+240}:{#4})$) {};
	\draw[edge] (s1) -- (s2) -- (s3) -- (s1);
	
	\makesmallgul{#1 - 0.5 - 0.433}{#2 + 0.5 + 1.118}{90}{0.5}
	\makesmallgdl{#1 - 0.5 - 0.433}{#2 - 0.5 - 1.118}{90}{0.5}
	\makesmallgr{#1 + 1.867}{#2}{90}{0.5}
}

\newcommand{\makegtwor}[4]{
    \node[vertex, color = white] (center) at ({#1},{#2}) {};
	\node[vertex] (s1) at ($(center) + ({#3}:{#4})$) {};
	\node[vertex] (s2) at ($(center)+({#3+120}:{#4})$) {};
	\node[vertex] (s3) at ($(center)+({#3+240}:{#4})$) {};
	\draw[edge] (s1) -- (s2) -- (s3) -- (s1);
	
	\makesmallgrul{#1 - 0.5 - 0.433}{#2 + 0.5 + 1.118}{90}{0.5}
	\makesmallgdl{#1 - 0.5 - 0.433}{#2 - 0.5 - 1.118}{90}{0.5}
	\makesmallgr{#1 + 1.867}{#2}{90}{0.5}
}

\newcommand{\makegtwoul}[4]{
    \node[vertex, color = white] (center) at ({#1},{#2}) {};
	\node[vertex] (s1) at ($(center) + ({#3}:{#4})$) {};
	\node[vertex] (s2) at ($(center)+({#3+120}:{#4})$) {};
	\node[vertex] (s3) at ($(center)+({#3+240}:{#4})$) {};
	\draw[edge] (s1) -- (s2) -- (s3) -- (s1);
	
	\makesmallgul{#1 - 0.5 - 0.433}{#2 + 0.5 + 1.118}{90}{0.5}
	\makesmallguldl{#1 - 0.5 - 0.433}{#2 - 0.5 - 1.118}{90}{0.5}
	\makesmallgr{#1 + 1.867}{#2}{90}{0.5}
}

\newcommand{\makegtwodl}[4]{
    \node[vertex, color = white] (center) at ({#1},{#2}) {};
	\node[vertex] (s1) at ($(center) + ({#3}:{#4})$) {};
	\node[vertex] (s2) at ($(center)+({#3+120}:{#4})$) {};
	\node[vertex] (s3) at ($(center)+({#3+240}:{#4})$) {};
	\draw[edge] (s1) -- (s2) -- (s3) -- (s1);
	
	\makesmallgul{#1 - 0.5 - 0.433}{#2 + 0.5 + 1.118}{90}{0.5}
	\makesmallgdl{#1 - 0.5 - 0.433}{#2 - 0.5 - 1.118}{90}{0.5}
	\makesmallgdlr{#1 + 1.867}{#2}{90}{0.5}
}

\newcommand{\makegthree}[3]{
	\makegtwoul{#1 - 0.5}{#2 + 4.1}{0}{1}
	\makegtwodl{#1 - 3.3}{#2 - 2.48}{0}{1}
	\makegtwor{#1 + 3.805}{#2 - 1.618}{0}{1}
	
	\node[vertex, color = white] (center) at ({#1},{#2}) {};
    \node[vertex] (t1) at ($(center) + ({#3}:2)$) {};
	\node[vertex] (t2) at ($(center)+({#3+120}:2)$) {};
	\node[vertex] (t3) at ($(center)+({#3+240}:2)$) {};
	\draw[edge] (t1) -- (t2) -- (t3) -- (t1);
}

\newcommand{\central}[4]{
    \node[vertex, color = white] (center) at ({#1},{#2}) {};
	\node[vertex, label=$a_1$] (s1) at ($(center) + ({#3}:{#4})$) {};
	\node[vertex, label=right:$a_2$] (s2) at ($(center)+({#3+120}:{#4})$) {};
	\node[vertex, label=right:$a_3$] (s3) at ($(center)+({#3+240}:{#4})$) {};
	\draw[edge] (s1) -- (s2) -- (s3) -- (s1);
	
	\node[added] (x1) at ($(center) + ({#3 - 60}:{0.5*#4})$) {};
	\node[added] (x2) at ($(x1) + ({#3 - 60}:{0.5*#4})$) {};
	\node[added] (x3) at ($(x2) + ({#3 - 30}:{0.5*#4})$) {};
	\node[added] (x4) at ($(x2) + ({#3 - 90}:{0.5*#4})$) {};
	\draw[edge] (x1) -- (x2) -- (x3);
	\draw[edge] (x2) -- (x4);
	
	\node[added] (x5) at ($(center) + ({#3 + 60}:{0.5*#4})$) {};

}

\newcommand{\makegtwolarge}[4]{
    \central{#1}{#2}{#3}{#4}
	\node[added] (x6) at ($(s3) + ({#3 + 120}:{0.5*#4})$) {};
	\node[added] (x7) at ($(x6) + ({#3 + 180}:{0.5*#4})$) {};
	\draw[edge] (s3) -- (x6) -- (x7);
	
	\makeg{#1 - 0.5 - 0.651}{#2 + 0.5 + 1.491}{90}{0.75}
	
	\node[added] (x8) at ($(a1) + ({#3 + 120}:{0.5*#4})$) {};
	\node[added] (x9) at ($(x8) + ({#3 + 180}:{0.5*#4})$) {};
	\node[added] (x10) at ($(x9) + ({#3 + 60}:{0.5*#4})$) {};
	\draw[edge] (a1) -- (x8) -- (x9);
	\draw[edge] (x8) -- (x10);

	\makeg{#1 - 0.5 - 0.651}{#2 - 0.5 - 1.491}{90}{0.75}
	\makeg{#1 + 2.3}{#2}{90}{0.75}
	
	\node[added] (y1) at ($(b) + ({#3 + 60}:{0.66*#4})$) {};
	\node[added] (y2) at ($(y1) + ({#3 + 90}:{0.9*#4})$) {};
	\node[added] (y3) at ($(y2) + ({#3 + 120}:{0.66*#4})$) {};
	\node[added] (y4) at ($(y3) + ({#3 + 60}:{0.66*#4})$) {};
	\draw[edge] (y1) -- (y2) -- (y3) -- (y4);
}

\title{Computing graph gonality is hard}

\author[1]{Dion Gijswijt}
\author[2]{Harry Smit}
\author[2,3]{Marieke van der Wegen}

\affil[1]{Delft University of Technology}
\affil[2]{Mathematical Institute, Utrecht University, PO Box 80.010, 3508 TA Utrecht, The Netherlands}
\affil[3]{Department of Information and Computing Sciences, Utrecht University, Princetonplein 5, 3584 CC Utrecht,  The Netherlands}

\begin{document}
\maketitle

\begin{abstract}
There are several notions of gonality for graphs. 
The divisorial gonality $\dgon(G)$ of a graph $G$ is the smallest degree of a divisor of positive rank in the sense of Baker-Norine. The stable gonality $\sgon(G)$ of a graph $G$ is the minimum degree of a finite harmonic morphism from a refinement of $G$ to a tree, as defined by Cornelissen, Kato and Kool. We show that computing $\dgon(G)$ and $\sgon(G)$ are NP-hard by a reduction from the maximum independent set problem and the vertex cover problem, respectively. Both constructions show that computing gonality is moreover APX-hard.
\end{abstract}

\section{Introduction}
In complex geometry, one attaches to a compact Riemann surface (equivalently, a complex algebraic curve) an invariant called gonality. This invariant measures `how far' a given Riemann surface $X$ is from the Riemann sphere $\widehat{\mathbf{C}}$; gonality is the minimal degree of a holomorphic map $X \rightarrow \widehat{\mathbf{C}}$. Alternatively, gonality can be defined as the minimal degree of a divisor of rank one. As we will now explain, both of these definitions can be transferred from Riemann surfaces to graphs, where, maybe somewhat surprisingly, the corresponding notions are no longer equivalent.

\paragraph{Divisorial gonality}
The first notion of gonality of graphs was introduced by Baker and Norine \cite{BN2007}, who developed a theory of divisors on finite graphs in which they uncovered many parallels between finite graphs and Riemann surfaces. In particular, they stated and proved a graph theoretical analogue of the classical Riemann-Roch theorem. See \cite{Baker2008, CornelissenKatoKool, HKN} for background on the interplay between divisors on graphs, curves and tropical curves. 

As observed in \cite{BN2007}, there is also a close connection between divisor theory and the chip-firing game of Bj\"orner, Lov\'asz and Shor \cite{BLS1991}. A divisor can be thought of as a distribution of chips over the vertices of a graph, where every vertex is assigned an integer number of chips. 
One divisor is transformed into another by firing sets of vertices: when a set $U$ is fired, a chip is moved along each edge from $U$ to $V\setminus U$. Bj\"orner, Lov\'asz and Shor considered the game where vertices are assigned non-negative numbers of chips and studied whether there is an infinite sequence of singleton sets that can be fired without any vertex getting a negative number of chips. 
See \cite{Merino} for the connections to the Abelian sandpile model from statistical physics and Biggs' dollar game \cite{Biggs}.  

The divisorial gonality $\dgon(G)$ of a connected graph $G$ is an important parameter associated to $G$ in the context of divisor theory. 
It is defined as the smallest degree of a positive rank divisor. In terms of the chip-firing game, the degree is the number of chips in the game. Positive rank means the following: for every vertex $v$ the divisor can be transformed into a divisor which assigns at least one chip to $v$ and a non-negative number to all other vertices.  
From \cite[Corollary 5.4]{BN2007} it follows that the divisorial gonality of a connected graph $G$ (with at least 3 vertices) is related to the finiteness of a game in the sense of \cite{BLS1991}. Specifically, $\dgon(G)=2|E|-|V|-t(G)$, where $t(G)$ is the maximum number of chips that can be placed on the graph $G=(V,E)$ such that adding a chip at an arbitrary vertex still results in a finite game. 

\paragraph{Stable gonality}
The second definition of gonality of Riemann surfaces was translated to graphs by Cornelissen, Kato and Kool \cite{CornelissenKatoKool} and is called stable gonality. Recall that gonality is defined using morphisms to the Riemann sphere, which is the unique compact Riemann surface with first Betti number zero. As the graphs with first Betti number (also known as circuit rank) zero are exactly the trees, stable gonality is defined using morphisms to trees. The morphisms considered are finite harmonic morphisms. Intuitively, these are morphisms that divide the edges of the graph equally over the edges of the tree. This notion is called stable because we are allowed to \emph{refine} the graph first: a refinement of a graph is obtained by adding degree one vertices and subdividing edges. The stable gonality $\sgon(G)$ of a connected graph $G$ is the minimum degree of such a finite harmonic morphism from a refinement of $G$ to a tree. See \cite{Amini2, Caporaso} for similar notions of gonality on tropical curves and graphs. 

The gonality of a Riemann surface defined over a number field (considered as a compact Riemann surface over $\mathbb{C}$) is bounded from below by the stable gonality of the intersection dual graph of its reduction modulo any prime ideal of the number field. This makes stable gonality of graphs relevant for number theoretic problems (e.g.\ \cite{CornelissenKatoKool}). 

\paragraph{Computational complexity}
As far as we know, no efficient algorithm exists to compute the gonality of an arbitrary Riemann surface. 
In \cite{SchichoSchreyerWeimann}, a good algorithm is given to compute the gonality of Riemann surfaces of small genus, but for arbitrary curves of genus larger than 7 the computations become too involved for the algorithm to terminate in reasonable time. In this paper we show that both notions of gonality of graphs are NP-hard to compute, so we cannot expect efficient algorithms to compute them either (unless P = NP). An easier problem is to decide whether a given curve or graph is hyperelliptic, i.e.\ has gonality 2: there is an algorithm in Magma for this problem on curves \cite{magma}, but a rigorous analysis of its complexity has not been carried out. However, there are algorithms that decide this for both notions of gonality for graphs in quasilinear time \cite{BBCW}. 

The computational complexity of the notions of gonality for graphs is also interesting from an algorithmic point of view, where one can ask whether these new graph parameters can be used for fixed parameter tractable algorithms. It is known that treewidth is a lower bound for both notions of graph gonality \cite{GonTW}, which raises the question whether or not NP-hard problems exist that are not tractable on graphs of bounded treewidth, but are tractable on graphs of bounded gonality. 

It is known that computing the divisorial gonality of a graph is in the complexity class XP: for every divisor with $k$ chips, we can check whether it has positive rank in polynomial time \cite{BakerShokrieh}. There exists an algorithm to compute the stable gonality of a graph in $O((1.33n)^nm^m\text{poly}(n,m))$ time, where $n$ is the number of vertices and $m$ the number of edges of the graph \cite{GKW}.

\paragraph{Bounds on gonality}
An upper bound for stable gonality has been established in terms of the first Betti number $b_1 = |E| - |V| + 1$: for any connected graph $G$ one has
\[
\sgon(G) \leq \frac{b_1 + 3}{2},
\]
matching the classical Brill-Noether bound \cite{CornelissenKatoKool}. Although this is a graph theoretic statement, no combinatorial proof is known for this upper bound. For divisorial gonality the same upper bound is conjectured, see \cite{Baker2008}. This is better than the trivial upper bound $\dgon(G) \leq |V|$. 

\begin{figure}
    \centering
    \begin{tikzpicture}
    \node (tw) at (0,0) {$\tw$};
    \node (sd) at (0,-1) {$\sdgon$};
    \node[text depth = .1cm, text
    height = .2cm] (d) at (1,-2) {$\dgon$};
    \node[text depth = .1 cm, text
    height = .2cm] (s) at (-1,-2) {$\sgon$};
    \node (b) at (0,-3) {$\frac{b_1+3}{2}$};
    \draw[->] (tw) -- (sd);
    \draw[->] (sd) to [bend left] (d);
    \draw[->, dashed] (d) to [bend left] (b);
    \draw[->] (sd) to [bend right] (s);
    \draw[->, dotted] (s) to [bend right] (b);
    
    \draw[->] (3.3,0) -- (4,0);
    \draw[->, dashed] (3.3,-.5) -- (4,-.5);
    \draw[->, dotted] (3.3,-1) -- (4,-1);
    
    \node[anchor=west, text depth = .1cm, text
    height = .2cm] (d) at (4,0) {\small bounded above by};
    \node[anchor=west, text depth = .1cm, text
    height = .2cm] (s) at (4,-.5) {\small conjecturally bounded above by};
    \node[anchor=west, text depth = .1cm, text
    height = .2cm] (b) at (4,-1) {\small bounded above by,};
    \node[anchor=west, text depth = .1cm, text
    height = .2cm] (b) at (4,-1.5) {\small no combinatorial proof known};
    \draw (3,.3) -- (9,.3) -- (9,-1.8) -- (3,-1.8) -- (3,.3);
    \end{tikzpicture}
    \caption{An overview of the relations between different graph invariants.}
    \label{fig:relation2}
\end{figure}
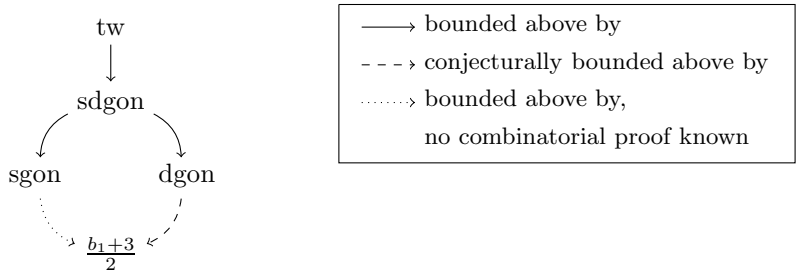
Bounds between the different invariants (including a hybrid version called “stable divisorial gonality sdgon(G)”, defined in 2.4 below) are represented in Figure 1, with the convention that when \emph{no} arrow is drawn from one parameter to another, the one is unbounded in terms of the other. 
In Section \ref{sec:untied} we elaborate on this and prove that divisorial and stable gonality are unbounded in one another. 

In \cite{AminiKool2014}, a lower bound $\dgon(G)\geq\frac{|V|\lambda_1}{24\Delta}$ is given in terms of the smallest nonzero eigenvalue $\lambda_1$ of the Laplacian $Q(G)$ and the maximum degree $\Delta$ of $G$. For stable gonality we have $\sgon(G)\geq\frac{|V|\lambda_1}{\lambda_1 + 4(\Delta + 1)}$ \cite[Theorem 5.10]{CornelissenKatoKool}.

\section{Definitions and notation}
\label{sec:prelims}

\subsection{Graphs and divisors}
Throughout the paper we consider only connected graphs $G=(V,E)$. We allow graphs to have parallel edges and loops. We write $u-v$-path for a path from $u$ to $v$. By \emph{path} we always mean simple path; we call a non-simple path a \emph{walk}.

For $A,B\subseteq V$, we denote by $E(A,B)$ the set of edges with an end in $A$ and an end in $B$ and by $E[A]:=E(A,A)$ the set of edges with both ends in $A$. For vertices $u,v\in V$, we use the abbreviation $E(u,v):=E(\{u\},\{v\})$ for the set of edges between $u$ and $v$ and we write $E(u):=E(\{u\}, V\setminus \{u\})$ for all edges incident to $u$. By $\deg(u)$ we denote the degree of a vertex $u$, where loops are counted twice. The \emph{Laplacian} of $G$ is the matrix $Q(G)\in \mathbb{Z}^{V\times V}$ defined by  
\begin{equation}
Q(G)_{uv}:=\begin{cases}\deg(u) - 2 |E(u,u)| &\text{if $u=v$},\\-|E(u,v)|&\text{otherwise}.\end{cases}
\end{equation}

A vector $D\in \Z^V$ is called a \emph{divisor} on $G$ and $\deg(D):=\sum_{v\in V}D(v)$ is its \emph{degree}. A divisor $D$ is \emph{effective} if $D\geq 0$, i.e.\ $D(v) \geq 0$ for all $v\in V$. Two divisors $D$ and $D'$ are \emph{equivalent}, written $D\sim D'$, if there is an integer vector $x\in \Z^V$ such that $D-D'=Q(G)x$. This is indeed an equivalence relation and equivalent divisors have equal degrees as the entries in every column of $Q(G)$ sum to zero.

Let $D$ be a divisor. If $D$ is equivalent to an effective divisor, then we define
\begin{align*}
\rank(D):=\max\{k\mid & \text{ $D-E$ is equivalent to an effective divisor} \\
	&\text{ for every effective $E$ of degree $\leq k$}\}.
\end{align*}
If $D$ is not equivalent to an effective divisor, we set $\rank(D):=-1$. Observe that equivalent divisors have the same rank. Answering a question of H.W. Lenstra, it was shown in \cite{RankHard} that computing the rank of a divisor is NP-hard. 

Finally, we define the \emph{divisorial gonality} of $G$ to be 
\begin{equation}
\dgon(G):=\min\{\deg(D)\mid \rank(D)\geq 1\}.
\end{equation}

Observe that in the definition of divisorial gonality, we can restrict ourselves to \emph{effective} divisors $D$. Hence $\dgon(G)$ is the minimum degree of an effective divisor $D$ such that for every vertex $v$ there is an effective divisor $D'\sim D$ with $D'(v)\geq 1$. Also note that $\dgon(G)\leq |V|$ since taking $D(v)=1$ for all $v\in V$ gives a divisor of positive rank.

To facilitate reasoning about equivalence of effective divisors, we denote by $\one_U$ the incidence vector of a subset $U$ of $V$. When $D$ and $D'$ are effective divisors and $D'=D-Q(G)\one_U$, we say that $D'$ is obtained from $D$ by firing the set $U$. If we think of $D(v)$ as the number of chips on a vertex $v$, then firing $U$ corresponds to moving one chip along each edge of the cut $E(U,V\setminus U)$ in the direction from $U$ to $V\setminus U$. In particular, we must have that $D(v)\geq |E(\{v\},V\setminus U)|$ for every $v\in U$ as $D'\geq 0$. Hence, we cannot fire any set $U$ for which the cut $E(U,V\setminus U)$ has more than $\deg(D)$ edges.

The following lemma from \cite{GonTW} shows that for equivalent effective divisors $D$ and $D'$, we can obtain $D'$ from $D$ by successively firing sets.

\begin{lemma}\label{chain}
Let $D$ and $D'$ be equivalent effective divisors satisfying $D \neq D'$. Then there is a chain of sets $\emptyset\subsetneq U_1\subseteq U_2\subseteq \cdots\subseteq U_k\subsetneq V$ such that $D_t:=D-Q(G) \sum_{i=1}^t \one_{U_i}$ is effective for every $t=1,\ldots,k$ and such that $D_k=D'$.
\end{lemma}

A divisor $D$ is called $v$-reduced, if $D(u) \geq 0$ for all $u\in V\setminus\{v\}$ and for every set $U\subseteq V\setminus\{v\}$ there is a vertex $u\in V\setminus\{v\}$ such that $D'(u) <0$, where $D'=D-Q(G)\one_U$. That is, $D$ is $v$-reduced if $D$ is effective outside $v$ and no subset of $V\setminus \{v\}$ can be fired. 

\begin{lemma}[{\cite[Lemma 1.4]{GonTW}}]\label{lem:unique-v-reduced}
Let $D$ be a divisor and $v\in V$ a vertex. There is a unique $v$-reduced divisor equivalent to $D$. 
\end{lemma}

Notice that a divisor $D$ has rank at least 1 if and only if for every $v$-reduced divisor it holds that $D(v) \geq 1$. 

We end this section with the notion of \emph{refinements} and \emph{stable divisorial gonality}.

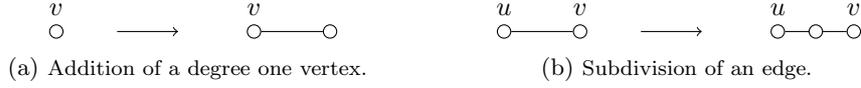
\begin{figure}
	\centering
	\begin{subfigure}{.4\textwidth}
		\centering
		\begin{tikzpicture}
		\node[vertex, label=$v$] (v) at (0,0) {};
		\draw[->] (.8,0) -- (1.6,0);
		\node[vertex, label=$v$] (v) at (2.6,0) {};
		\node[vertex] (u) at (3.6,0) {};
		\draw[edge] (u) -- (v);
		\end{tikzpicture}
		\caption{\footnotesize Addition of a degree one vertex.}
		\label{fig:leaf}
	\end{subfigure}	
	\quad\begin{subfigure}{.4\textwidth}
		\centering
		\begin{tikzpicture}
		\node[vertex, label=$u$] (u) at (0,0) {};
		\node[vertex, label=$v$] (v) at (1,0) {};
		\draw[edge] (u) -- (v);
		\draw[->] (1.8,0) -- (2.6,0);
		\node[vertex, label=$u$] (v) at (3.6,0) {};
		\node[vertex, label=$v$] (u) at (4.6,0) {};
		\draw[edge] (u) -- (v);
		\node[vertex] (w) at (4.1,0) {};
		\end{tikzpicture}
		\caption{\footnotesize Subdivision of an edge.}
		\label{fig:subdivision}
	\end{subfigure}
	\caption{Refinements.}\label{fig:refinement}
\end{figure}

\begin{definition}\label{def:refinement}
Let $G$ be a graph. A \emph{refinement} of $G$ is a graph $G'$ obtained from $G$ by subdividing edges and adding degree one vertices, see Figure \ref{fig:refinement}.
\end{definition}

\begin{definition}\label{def:sdgon}
Let $G$ be a graph. The \emph{stable divisorial gonality} of $G$ is defined as 
\begin{align*}
\sdgon(G) = \min\{\dgon(G') \mid G' \text{ a refinement of G}\}.
\end{align*}
\end{definition}

\begin{remark}
Since adding leaves to a graph does not change its divisorial gonality, we do not have to consider all refinements, but only those obtained from subdivisions of edges, i.e.\
\begin{align*}
\sdgon(G) = \min\{\dgon(G') \mid G' \text{ a subdivision of G}\}.
\end{align*}
\end{remark}

\subsection{Morphisms and stable gonality}

The stable gonality of a graph is defined using finite harmonic morphisms. A finite harmonic morphism is a graph homomorphism with some extra properties. Recall that a graph homomorphism is a map $\phi\colon G\to H$ that maps vertices of $G$ to vertices of $H$ and preserves edges, i.e., a homomorphism is a map $\phi\colon V(G) \cup E(G) \to V(H)\cup E(H)$ such that 
\begin{itemize}
	\item $\phi(V(G)) \subseteq V(H)$;
	\item $\phi(E(u,v)) \subseteq E(\phi(u)\phi(v))$ for all pairs of vertices $u,v\in V(G)$.
\end{itemize}

\begin{definition}
Let $G$ and $H$ be loopless graphs. A \emph{finite morphism} is a graph homomorphism $\phi\colon G\to H$, together with a map $r\colon E(G) \to \Z_{>0}$ that assigns an index $r(e)$ to every edge $e \in E(G)$.  
\end{definition}

Intuitively, a finite morphism is harmonic if it divides the edges of $G$ equally over the edges of $H$. We make this precise in the following definitions. 

\begin{definition}
Let $G = (V,E)$ and $H = (W,F)$ be loopless graphs and $\phi\colon G\to H$ a finite morphism. Let $v \in V$ be a vertex of $G$ and let $f\in F$ be an edge of $H$ that is incident to $\phi(v)$. The \emph{index of $v$ in the direction of $f$}, denoted by $m_{\phi, f}(v)$, is 
\begin{align*}
m_{\phi, f}(v) = \sum_{e\in E(v), \phi(e) = f} r(e).
\end{align*}
\end{definition}
\begin{definition}
Let $G$ and $H$ be loopless graphs and $\phi\colon G\to H$ a finite morphism. Then $\phi$ is \emph{harmonic} if for every vertex $v$ of $G$ we have $m_{\phi, f}(v) = m_{\phi, f'}(v)$ for all edges $f$ and $f'$ of $H$ incident to $\phi(v)$. We abbreviate $m_\phi(v) := m_{\phi,f}(v)$ for any $f$. 
\end{definition}

\begin{definition}
Let $G=(V,E)$ and $H=(W,F)$ be loopless graphs and $\phi\colon G\to H$ a finite harmonic morphism. Let $f\in F$ be an edge of $H$. The \emph{degree} $\deg(\phi)$ is \begin{align*}
\deg(\phi) = \sum_{e\in E, \phi(e) = f} r_\phi(e). 
\end{align*}
This is independent of the choice of $f$ \cite[Lemma 2.4]{BN2009}. This is also equal to \begin{align*}
\deg(\phi) = \sum_{v\in V, \phi(v) = w} m_\phi(v), 
\end{align*} for any vertex $w\in W$ of $H$.
\end{definition}

We now turn to the definition of stable gonality. Recall the notion of a refinement from Definition~\ref{def:refinement}. 

\begin{definition}
The \emph{stable gonality} $\sgon(G)$ of a graph $G$ is \begin{align*}
\sgon(G) = \min\{\deg(\phi) \mid &\  \phi\colon G' \to T \text{ a finite harmonic morphism,}\\
&\text{ where $T$ is a tree and $G'$ is a refinement of $G$}\}.
\end{align*}
\end{definition}

Notice that finite morphisms are not defined for graphs that contain loops, but since we obtain a loopless graph by subdividing all edges, stable gonality is defined for graphs with loops. 

\begin{remark}
	The stable gonality of a disconnected graph equals the sum of the stable gonality of all components. In the remainder of this paper, we only consider connected graphs. 
\end{remark}

\begin{example}
Let $G$ be a tree. Set $r(e) = 1$ for every edge $e$ and consider the identity map $\phi\colon G \to G$. This is a finite harmonic morphism of degree 1. Thus the stable gonality of a tree equals 1. In fact, trees are the only graphs with stable gonality 1. Indeed, a finite harmonic morphism of degree 1 is injective  
and we cannot map a graph that contains a cycle injectively to a tree. 
\end{example}
\begin{example}
Let $G$ be a cycle with $2n$ vertices. In the previous example we have seen that $\sgon(G) \geq 2$. We give an morphism of degree 2 to show that $\sgon(G) = 2$. Let $T$ be a path on $n+1$ vertices. Assign index 1 to all edges of $G$ and consider the map in Figure \ref{fig:example-cycle-sgon}. This is a finite harmonic morphism of degree 2, hence the stable gonality of an even cycle is 2. If $G$ is an odd cycle, we obtain an even cycle by subdividing one of its edges. We use the same morphism of degree 2 to show that $\sgon(G) = 2$. 
\end{example}

\begin{figure}
	\centering
	\begin{tikzpicture}
	\node[vertex] (v1) at (0,0) {};
	\node[vertex] (v2) at (.8,.8) {};
	\node[vertex] (v3) at (.8,2) {};
	\node[vertex] (v4) at (0,2.8) {};
	\node[vertex] (v5) at (-.8,2) {};
	\node[vertex] (v6) at (-.8,.8) {};
	\draw[edge] (v1) -- (v2) -- (v3) -- (v4) -- (v5) -- (v6) -- (v1);
	\draw[->] (1.4,1.4) -- (2.2,1.4);
	\node[vertex] (v1) at (2.6,0) {};
	\node[vertex] (v2) at (2.6,.8) {};
	\node[vertex] (v3) at (2.6,2) {};
	\node[vertex] (v4) at (2.6,2.8) {};
	\draw[edge] (v1) -- (v2) -- (v3) -- (v4);
	\end{tikzpicture}
	\caption{A finite harmonic morphism from a cycle to a path. Every vertex is mapped to the vertex on its right side and every edge is assigned index $1$. }\label{fig:example-cycle-sgon}
\end{figure}
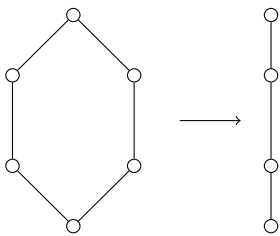

\section{Divisorial gonality is hard}\label{sec:dgon_hard}
\subsection{NP-hardness}

We define the \textsc{Divisorial Gonality} problem as follows: 
\begin{verse}
\textsc{Divisorial Gonality}\\
\textbf{Input:} Graph $G = (V,E)$, integer $k \leq |V|$.\\
\textbf{Question:} Is $\dgon(G) \leq k$?
\end{verse}

We prove that this problem is NP-hard by a reduction from the \textsc{Independent Set} problem. Given an instance $(G,k)$ of the \textsc{Independent Set} problem, we construct a graph $\widehat{G}$. We give a precise relation between the divisorial gonality of $\widehat{G}$ and the maximum size of an independent set of $G$. The constructed graph $\widehat{G}$ contains many parallel edges, and in the proof we use the fact that we either need many chips to move along these edges, or we will never move chips along these edges. 

Let $G=(V,E)$ be a graph and define $M:=3|V|+2|E|+2$. We construct the graph $\widehat{G}$ in the following way. Start with a single vertex $T$. For every $v\in V$ add three vertices: $v,v',T_v$. For every edge $e\in E(u,v)$, add two vertices $e_u$ and $e_v$. The edges of $\widehat{G}$ are as follows. For every $e\in E(u,v)$, add an edge between $e_u$ and $e_v$, add $M$ parallel edges between $u$ and $e_u$ and $M$ parallel edges between $e_v$ and $v$. For every $v\in V$, add three parallel edges between $v'$ and $T_v$, $M$ parallel edges between $v$ and $v'$ and $M$ parallel edges between $T_v$ and $T$. See Figure \ref{F1} for an example.

\begin{figure}[h]
	\centering
		\begin{tikzpicture}[scale=0.8]
		\node[vertex, label=below left:$u$] (u) at (0,0) {};
		\node[vertex, label=below right:$v$] (v) at (3,0) {};
		\node[vertex] (w) at (3,3) {};
		\node[vertex] (x) at (0,3) {};
		\draw[edge] (u)-- node[below] {$e$} (v);
		\draw[edge] (v)--(w)--(x)--(u);
		\draw[edge] (u)--(w);
		\end{tikzpicture}
	\qquad\qquad
		\begin{tikzpicture}[scale=0.8]
		\node[vertex, label=below left:$u$] (u) at (0,0) {};
		\node[vertex, label=below right:$v$] (v) at (3,0) {};
		\node[vertex] (w) at (3,3) {};
		\node[vertex] (x) at (0,3) {};
		\node[added, label=above:$T$] (T) at (1.5,4.5) {};

		\node[added, label=below:$e_u$] (eu) at (1,0) {};
		\node[added, label=below:$e_v$] (ev) at (2,0) {};
		\draw[fatedge] (u)--(eu);
		\draw[fatedge] (ev)--(v);
		\draw[edge] (ev)--(eu);				
		\node[added] (fv) at (3,1) {};
		\node[added] (fw) at (3,2) {};
		\draw[fatedge] (v)--(fv);
		\draw[fatedge] (fw)--(w);
		\draw[edge] (fv)--(fw);				
		\node[added] (gw) at (2,3) {};
		\node[added] (gx) at (1,3) {};
		\draw[fatedge] (w)--(gw);
		\draw[fatedge] (gx)--(x);
		\draw[edge] (gw)--(gx);				
		\node[added] (hx) at (0,2) {};
		\node[added] (hu) at (0,1) {};
		\draw[fatedge] (x)--(hx);
		\draw[fatedge] (u)--(hu);
		\draw[edge] (hx)--(hu);				
		\node[added] (iu) at (0.7,0.7) {};
		\node[added] (iw) at (2.3,2.3) {};
		\draw[fatedge] (u)--(iu);
		\draw[fatedge] (w)--(iw);
		\draw[edge] (iu)--(iw);	
		
		\node[added, label=below right:$v'$] (v') at (4,1) {};			
		\node[added, label=right:$T_v$] (Tv) at (4,2) {};			
		\draw[fatedge] (v)--(v');
		\draw[fatedge] (Tv) to[out=90, in=0] (T);
		\draw[edge] (v')--(Tv);		
		\draw[edge] (v') to[relative, out=40, in=140] (Tv);		
		\draw[edge] (v') to[relative, out=-40, in=-140] (Tv);		

		\node[added, label=below left:$u'$] (u') at (-1,1) {};			
		\node[added, label=left:$T_u$] (Tu) at (-1,2) {};			
		\draw[fatedge] (u)--(u');
		\draw[fatedge] (Tu) to[out=90, in=180] (T);
		\draw[edge] (u')--(Tu);		
		\draw[edge] (u') to[relative, out=40, in=140] (Tu);		
		\draw[edge] (u') to[relative, out=-40, in=-140] (Tu);		

		\node[added] (w') at (2.5,3.5) {};
		\node[added] (Tw) at (2,4) {};			
		\draw[fatedge] (w)--(w');
		\draw[fatedge] (Tw)--(T);
		\draw[edge] (w')--(Tw);		
		\draw[edge] (w') to[relative, out=40, in=140] (Tw);		
		\draw[edge] (w') to[relative, out=-40, in=-140] (Tw);		
		\node[added] (x') at (0.5,3.5) {};
		\node[added] (Tx) at (1,4) {};			
		\draw[fatedge] (x)--(x');
		\draw[fatedge] (Tx)--(T);
		\draw[edge] (x')--(Tx);		
		\draw[edge] (x') to[relative, out=40, in=140] (Tx);		
		\draw[edge] (x') to[relative, out=-40, in=-140] (Tx);		
		\end{tikzpicture}
	\caption{On the left a graph $G$, on the right the corresponding $\widehat{G}$, where the $M$-fold parallel edges are drawn as bold edges. Here $M=12+10+2=24$.}\label{F1}
\end{figure}
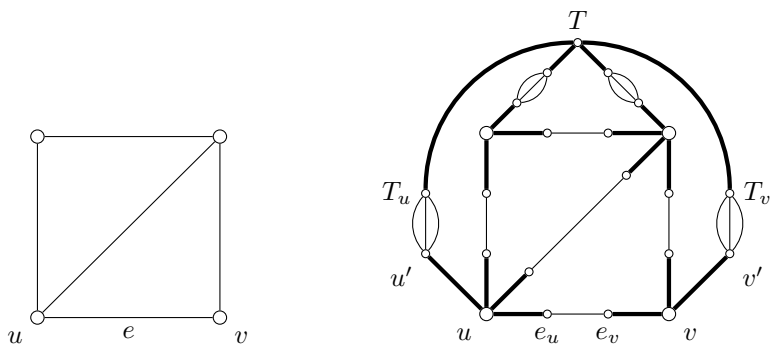

For an effective divisor $D$ on $G$, we define the following equivalence relation $\equiv_{D}$ on $V$:
\begin{equation}
u\equiv_D v \iff \text{$x_u=x_v$ for every $x\in \Z^V$ for which $D-Q(G)x\geq 0$.}
\end{equation}
An edge $e$ for which the ends are equivalent is called \emph{$D$-blocking}. 
Intuitively, this means firing a set that contains exactly one of equivalent $u$ and $v$ results in a divisor that is negative for some vertex.  
Note that if $D$ and $D'$ are equivalent effective divisors, then $\equiv_D=\equiv_{D'}$. We need the following observations.

\begin{lemma}\label{equiv}
	Let $D\geq 0$ be a divisor on $G$ and let $u,v\in V$. Then $u\not\equiv_D v$ if and only if for some effective divisor $D'\sim D$ we can fire a subset $U$ with $u\in U$, $v\not\in U$. In particular, $u\equiv_D v$ if every $u$--$v$ cut has more than $\deg(D)$ edges.
\end{lemma}
\begin{proof}
	This follows directly from Lemma \ref{chain}.
\end{proof}

\begin{lemma}\label{blocking}
	Let $D\geq 0$ be a divisor on $G=(V,E)$. Let $F$ be the set of $D$-blocking edges and let $U$ be a component of the subgraph $(V,E\setminus F)$. Then for every effective divisor $D'\sim D$ we have $\sum_{u\in U} D'(u)=\sum_{u\in U}D(u)$.
\end{lemma}
\begin{proof}
	This holds as no chips can move along a $D$-blocking edge.
\end{proof}

\begin{lemma}\label{gonalpha}
Let $G=(V,E)$ be a graph and let $\widehat{G}$ be the graph as constructed above. Let $\alpha(G)$ be the maximum size of an independent set in $G$. Then the following holds:
$$
\dgon(\widehat{G})=4|V|+|E|+1-\alpha(G).
$$
\end{lemma}
\begin{proof}
We first show that $\dgon(\widehat{G})\geq 4|V|+|E|+1-\alpha(G)$. Let $D\geq 0$ be a divisor on $\widehat{G}$ with $\rank(D)\geq 1$ and $\deg(D)=\dgon(\widehat{G})$. Consider the equivalence relation $\equiv_D$ on $V(\widehat{G})$. Clearly, $\deg(D)\leq |V(\widehat{G})| = 3|V| + 2|E| + 1 < M$. Hence, by Lemma \ref{equiv}, the $M$-fold edges are $D$-blocking. It follows that $T$ is equivalent to all vertices $T_v, v\in V$, and every vertex $v\in V$ is equivalent to $v'$ and to all vertices $e_v$. By Lemma \ref{blocking}, the number of chips on $\{T\}$ and on each of the sets $\{v\}$, $\{v',T_v\}$ and $\{e_u,e_v\}$ is constant over all effective divisors $D'\sim D$ (for all $u,v\in V$ and $e\in E(u,v)$). Hence, from the fact that $\rank(D)\geq 1$ it follows that for every $u,v\in V$ and $e\in E(u,v)$ we must have
\begin{itemize}
\item $D(T)\geq 1$;
\item $D(v)\geq 1$;
\item $D(v')+D(T_v)\geq 2$ if $v\equiv_D T$ and $D(v')+D(T_v)\geq 3$ otherwise;
\item $D(e_u)+D(e_v)\geq 2$ if $u\equiv_D v$ and $D(e_u)+D(e_v)\geq 1$ otherwise.
\end{itemize}
Let $U_0=\{v\in V\mid v\sim_D T\}$ be the set of nodes in $V$ that are equivalent to $T$. Now let $U_0\cup U_1\cup\cdots\cup U_k$ be the partition of $V$ induced by $\equiv_D$. By the above, we see that 
\begin{align*}
\deg(D) &\geq 1+|V|+(3|V|-|U_0|)+(|E|+|E[U_0]|)\\
&=4|V|+|E|+1 + |E[U_0]|-|U_0|.
\end{align*} 
Since 
\begin{align}\label{alphaU0}
\alpha(G)\geq \alpha(G[U_0])\geq |U_0|-|E[U_0]|,  
\end{align}
we find that $\dgon(G)=\deg(D)\geq 4|V|+|E|+1-\alpha(G)$.

Now we show that equality can be attained. For this, let $S\subseteq V$ be an independent set in $G$ of size $\alpha(G)$. Consider the partition $V\cup\{T\}=U_0\cup U_1\cup\cdots\cup U_k$ where $U_0=S\cup \{T\}$ and $U_1,\ldots, U_k$ are singletons. Every edge of $G$ has endpoints in two distinct sets $U_i$ and $U_j$ with $i<j$. We orient the edge from $U_i$ to $U_j$. We now define the effective divisor $D$ on $\widehat{G}$ as follows. 
\begin{align*}
D(v)&=1\quad\text{for every $v\in V\cup\{T\}$},\\
D(v')=D(T_v)&=1\quad \text{for every $v\in S$},\\
D(T_v)&=3\quad\text{for every $v\in V\setminus S$},\\
D(v')&=0\quad\text{for every $v\in V\setminus S$},\\
D(e_u)&=1\quad\text{for every edge $e$ with tail $u$},\\
D(e_v)&=0\quad\text{for every edge $e$ with head $v$}.
\end{align*}
It is easy to check that $\deg(D)=4|V|+|E|+1-\alpha(G)$. Define $V_i:=U_i\cup\cdots\cup U_k$ for every $i=1,\ldots, k$ and let 
$$
W_i:=V_i\cup\{v'\mid v\in V_i\}\cup \{e_u\mid u\in V_i\text{ is end point of an edge $e$}\}.
$$
Observe that the cut in $\widehat{G}$ induced by the set $W_i$ consists of the edges $e_ue_v$ for every $u\in V\setminus V_i, v\in V_i, e\in E(u,v)$ and the triple edges from $T_u$ to $u'$ for every $u\in V\setminus V_i$. Hence, we can fire the complement of $W_i$. That is,  for every $i$, $D':=D+Q(\widehat{G})\one_{W_i}$ is effective.  Furthermore, $D'(v')\geq 1$ for every $v\in U_i$ and $D'(e_v)\geq 1$ for every edge $e$ with head $v$ in $U_i$. It follows that $D$ has positive rank.
\end{proof}

\begin{theorem}\label{MainTheorem}
	The \textsc{Divisorial Gonality} problem is NP-hard.
\end{theorem}
\begin{proof}
The result follows directly from the NP-hardness of the {\sc Independent Set} problem, Lemma \ref{gonalpha}, and the fact that the transformation uses polynomial time.
\end{proof}

We define the \textsc{Stable Divisorial Gonality} problem similarly:
\begin{verse}
	\textsc{Stable Divisorial Gonality}\\
	\textbf{Input:} Graph $G = (V,E)$, integer $k \leq |V|$.\\
	\textbf{Question:} Is $\sdgon(G) \leq k$?
\end{verse}
We slightly change the arguments above to prove that the \textsc{Stable Divisorial Gonality} problem is NP-hard as well. Let $\widehat{G}'$ be a subdivision of $\widehat{G}$.
Define a \emph{$D$-blocking path} as a path where every internal vertex has degree $2$ and whose ends are equivalent. For any divisor $D$ with $\deg(D) < M$, Lemma \ref{equiv} guarantees that $M$ parallel edges in $\widehat{G}$ are subdivided into $D$-blocking paths in $\widehat{G}'$. Furthermore, Lemma \ref{blocking} is still valid for $D$-blocking paths. Following the first part of the proof of Lemma \ref{gonalpha}, it follows that $\dgon(\widehat{G}') \geq 4|V|+|E|+1-\alpha(G) = \dgon(\widehat{G})$. Therefore $\sdgon(\widehat{G}) = \dgon(\widehat{G})$ holds, which proves the following. 
\begin{theorem}
	The stable divisorial gonality problem is NP-hard. 
\end{theorem}

\subsection{APX-hardness}
We consider the optimization variant of the \textsc{Divisorial Gonality} problem, where we ask for the divisor with minimum degree among all divisors with rank at least $1$. 
We will prove that this optimisation problem is APX-hard. For this, we use the reduction above, restricted to subcubic graphs, since \textsc{Independent Set} is APX-hard for subcubic graphs \cite{APX}. For APX-hardness we also need to be able to construct good independent sets from good divisors in polynomial time.
It follows that there is a PTAS reduction of `maximum independent set on cubic graphs'  to `finding a minimum degree divisor of positive rank'. Since the former problem is APX-hard, also the second problem is APX-hard. 
It remains open whether or not  finding a positive rank divisor of minimum degree is in APX.

\begin{lemma} \label{lem:apx-poly-time}
	Let $D\geq0$ be a divisor on $G$. For any two vertices $u,v \in V(G)$, we can determine whether $u \sim_D v$ in polynomial time. 
\end{lemma}
\begin{proof} Let $D_u$ and $D_v$ be the unique $u$-reduced and $v$-reduced divisors equivalent to $D$, respectively (see Lemma \ref{lem:unique-v-reduced}). Let $x$ be such that $D_u-D_v = Q(G)x$. Then $u$ and $v$ are equivalent if and only if $x_u=x_v$. Since we can compute $D_u$ and $D_v$ in polynomial time, this completes the proof. 
\end{proof}

\begin{lemma}\label{lem:apx-bound} Let $G = (V,E)$ be a subcubic graph and let $\widehat{G}$ be as constructed above.
Let $D\geq0$ be a divisor on $\widehat{G}$ of positive rank and degree $\deg(D) \leq (1+\epsilon) \dgon(\widehat{G})$. Then we can find in polynomial time an independent set $S$ in $G$ of size at least $(1-22\epsilon)\alpha(G)$.
\end{lemma}
\begin{proof} We may assume that $\deg(D) \leq |V(\widehat{G})|$ (otherwise replace $D$ by the divisor with one chip on every vertex).  

As in the proof of Lemma \ref{gonalpha}, we define $U_0$ to be the vertices in $V$ that are equivalent to $T$. We have
\begin{align*}
\deg(D) \geq 4|V| + |E| + 1 + |E[U_0]| - |U_0|.
\end{align*}

As $G$ is subcubic, we get
\begin{align*}
\deg(D) \geq (11/2) |V| + 1 + |E(U_0)| - |U_0|.
\end{align*}
We construct an independent set $S$ in $G$ of size at least $|U_0| -|E[U_0]|$ by starting with the set $U_0$ and then deleting an endpoint for every edge in $E[U_0]$. Thus there exists an independent set $S$ in $G$ of size
\begin{align*}
|S| &\geq |U_0|-|E[U_0]| \geq (11/2) |V| + 1 - \deg(D) \\
 &\geq (11/2) |V| + 1 - (1+\epsilon) \dgon(\widehat{G}) \\
&\geq (11/2) |V| + 1 - (1+\epsilon) ((11/2)|V| + 1 - \alpha(G))\\
&= -\epsilon ((11/2)|V| + 1 ) + (1+\epsilon)\alpha(G).
\end{align*}
Since $G$ is subcubic, we have $\alpha(G)\geq(1/4)|V|$. Hence, we have
\begin{align*}
|S| &\geq -\epsilon (22\alpha(G) +1) + (1+\epsilon)\alpha(G) \\
&= (1-21\epsilon)\alpha(G) -\epsilon\\
&\geq (1-22\epsilon)\alpha(G).
\end{align*}
Notice that by Lemma \ref{lem:apx-poly-time}, we can find $S$ in polynomial time. 
\end{proof}

\begin{theorem}
	The \textsc{Divisorial Gonality} problem is APX-hard. 
\end{theorem}
\begin{proof}
	In \cite{APX} it was shown that the \textsc{Independent Set} problem is APX-hard, even on subcubic graphs. From this and from Lemma  \ref{lem:apx-bound} the result follows. 
\end{proof}

We consider the optimization version of the \textsc{Stable Divisorial Gonality} problem as well. Again, we use the same proof to show that this is APX-hard. 

\begin{theorem}
	The \textsc{Stable Divisorial Gonality} problem is APX-hard. 
\end{theorem}

\section{Stable gonality is hard}\label{sec:sgon_hard}

\subsection{NP-hardness}
We define the \textsc{Stable Gonality} problem as follows:
\begin{verse}
	\textsc{Stable Gonality}\\
	\textbf{Input:} Graph $G = (V,E)$, integer $k \leq |V|$.\\
	\textbf{Question:} Is $\sgon(G) \leq k$?
\end{verse}

In this section we prove that the \textsc{Stable Gonality} problem is NP-hard. 

Let $G$ be a simple graph and write $n$ for its number of vertices. We construct a multigraph $\widehat{G}$ from $G$ in two steps:
\begin{enumerate}
	\item Add a vertex $c$ to the vertex set of $G$, and connect this vertex with exactly one vertex of $G$ (chosen arbitrarily). Then, add $n-1$ more vertices to the vertex set, and connect these vertices with $c$ (and no other vertices). We will call this process ``adding a star of size $n$ with centre $c$''. 
	\item Add a vertex $t$ to the vertex set, and for every other vertex $u$ (including the vertices added in the previous step) draw $n$ parallel edges between $t$ and $u$.
\end{enumerate}
This construction is illustrated in Figure \ref{fig:constructie}. 	

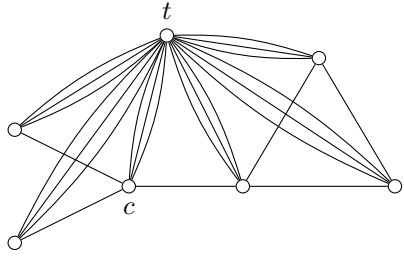
\begin{figure}
	\centering
	\begin{tikzpicture}
	\node[vertex] (u) at (0,0) {};
	\node[vertex] (v) at (2,0) {};
	\node[vertex] (w) at (1,1.7) {};
	\node[vertex, label=above:$t$] (t) at (-1,2) {};
	\node[vertex, label=below:$c$] (c) at (-1.5,0) {};
	\node[vertex] (a) at (-3,0.75) {};
	\node[vertex] (b) at (-3,-0.75) {};
	\draw[edge] (u) -- (v) -- (w) -- (u); 
	\draw[edge] (u) -- (c);
	\draw[edge] (a) -- (c) -- (b); 
	\draw[edge] (t) -- (a) (t) -- (b) (t) -- (c);
	\draw[edge] (t) -- (u) (t) -- (v) (t) -- (w);
	\draw[edge] (t) to[relative, in=170, out=10] (a);
	\draw[edge] (a) to[relative, in=170, out=10] (t);
	\draw[edge] (t) to[relative, in=170, out=10] (b);
	\draw[edge] (b) to[relative, in=170, out=10] (t);
	\draw[edge] (t) to[relative, in=170, out=10] (c);
	\draw[edge] (c) to[relative, in=170, out=10] (t);
	\draw[edge] (t) to[relative, in=170, out=10] (u);
	\draw[edge] (u) to[relative, in=170, out=10] (t);
	\draw[edge] (t) to[relative, in=170, out=10] (v);
	\draw[edge] (v) to[relative, in=170, out=10] (t);
	\draw[edge] (t) to[relative, in=170, out=10] (w);
	\draw[edge] (w) to[relative, in=170, out=10] (t);
	\end{tikzpicture}
	\caption{The graph $\widehat{G}$ when we start with a triangle $G$. }
	\label{fig:constructie}
\end{figure}

\begin{lemma}\label{lem:sgon-vertexcover}
	Let $G$ be a simple graph and $\widehat{G}$ the graph constructed above. Denote by $n$ the number of vertices of $G$ and by $\tau(G)$ the minimal size of a vertex cover of $G$. Then $\sgon(\widehat{G}) \leq n+k+1$ if and only if $\tau(G) \leq k$. 
\end{lemma}

\begin{proof}
	Suppose that $G$ has a vertex cover of size at most $k$, let $S$ be such a vertex cover. We will construct a finite harmonic morphism from $\widehat{G}$ to a tree.
	Define $T$ as a star with $2n-|S|$ vertices, write $t'$ for the centre and $v_1', \ldots, v'_{2n-|S|-1}$ for the other vertices. Now number the vertices in $V(\widehat{G}) \backslash(S \cup \{t,c\})$ from $v_1$ to $v_{2n-|S|-1}$. Define $\phi\colon V(\widehat{G}) \to T$:
	\begin{align*}
	\phi(v) = \begin{cases}
	t' &\text{if $v=t$, $v=c$ or $v\in S$,}\\
	v'_i & \text{if $v=v_i$.}
	\end{cases}
	\end{align*}
	See Figure \ref{fig:map} for an illustration.
	We refine $\widehat{G}$ and extend $\phi$ to this refinement in order to make $\phi$ a finite harmonic morphism. 
	\begin{enumerate}
		\item For all edges $e=uv \in E(\widehat{G})$ with $\phi(u) = \phi(v) = t'$, subdivide $uv$ with a vertex $w$ and add a leaf $l$ to $t'$. Define $\phi(w) = l$. If $u = t$, set $r_{\phi}(uw) = n$. Write $T'$ for the extended tree $T$ after this step. See Figure \ref{fig:step1}.
		\item For all vertices $v\in S\cup\{c,t\}$ and for all leaves $l\in V(T')$, if $v$ has no neighbour that is mapped to $l$, add a leaf $v_l$ to $v$ and set $\phi(v_l) = l$. If $v = t$, set $r_\phi(vv_l) = n$. Write $\widehat{G}'$ for the refinement of $\widehat{G}$ after this step. See Figure \ref{fig:step2}. 
		\item For all edges $e \in E(\widehat{G}')$, if we did not mention $r_{\phi}(e)$ explicitly before, we set $r_{\phi}(e) = 1$. 
	\end{enumerate}
	We prove that $\phi\colon \widehat{G}' \to T'$ is a finite harmonic morphism. It is clear that $\phi$ is a finite morphism. For harmonicity, we check all vertices $v\in V(\widehat{G}')$: 
	\begin{itemize}
		\item Suppose that $v = t$. Let $e \in E(T')$. \begin{itemize}
			\item Suppose that $e = t'v_i'$, then $m_{\phi,e}(v) = n$, since the edges $tu$ with $\phi(u) = v_i'$ are precisely the $n$ edges $tv_i$ all with index $1$. 
			\item Assume $e = t'l$ for some leaf $l$ that is added in step 1. Then there is exactly one edge $tu$ mapped to $e$, where $u$ is a vertex that is added in either step 1 or 2. This edge has index $n$, so $m_{\phi,e}(v) = n$. 
		\end{itemize}
		We conclude that $\phi$ is harmonic at $t$ and $m_\phi(t) = n$. 
		\item Suppose that $v \in S \cup \{c\}$. Let $e\in E(T')$. 
		\begin{itemize}
			\item Assume there is a neighbour $u$ of $v$ with $u$ a vertex of $\widehat{G}$ or a vertex that is added in step 1, such that $\phi(uv) = e$. By the construction of $\phi$ this is the only edge incident to $v$ that is mapped to $e$, and this edge has index $1$. 
			\item Suppose that there is no neighbour $u$ of $v$ with $u$ a vertex of $\widehat{G}$ or a vertex that is added in step 1, such that $\phi(uv) = e$. Then we added a leaf to $v$ in step 2 and assigned index $1$ to this new edge.  
		\end{itemize} 
		So $\phi$ is harmonic at $v$ and $m_{\phi}(v)=1$. 
		\item If $v$ is any other vertex, then $\phi(v)$ is a leaf, thus $\phi$ is harmonic at $v$. 
	\end{itemize}
	Now we compute the degree of $\phi$. For this we look at $t'$. The vertices that are mapped to $t'$ are $t$, $c$ and all vertices in $S$. It follows that $\deg(\phi) = m_{\phi}(t) + m_\phi(c) + \sum_{v\in S} m_\phi(v) = n + 1 + |S| \leq n+ 1 + k$. So $\sgon(\widehat{G}) \leq n+k+1$. 
	
	Now suppose that $\widehat{G}$ has stable gonality at most $n + k + 1$, i.e.\ there exists a refinement $\widehat{G}'$ of $\widehat{G}$, a tree $T$ and a map $\phi\colon \widehat{G}' \to T$ with degree $d \leq n + k + 1$. 
	
	Consider the graph $T \backslash \phi(t)$; as $T$ is a tree, $T \backslash \phi(t)$ is a forest. Define $C$ as a component of $T \backslash \phi(t)$ for which the size of $\phi^{-1}(C) \cap \widehat{G}$ is maximal, and let $\gamma = \big| \phi^{-1}(C) \cap \widehat{G} \big|$.
	
	Denoting by $e_C$ the unique edge between $\phi(t)$ and $C$, the degree of $\phi$ is at least the number of edges of $\widehat{G}'$ that are mapped to $e_C$. Because $t$ is connected to each vertex in $\widehat{G}$ with exactly $n$ edges, $t$ and $\phi^{-1}(C) \cap \widehat{G}$ are connected by $\gamma n$ edges (in $\widehat{G}$). Each such edge is subdivided into a path in $\widehat{G}'$, which must then be mapped to a path from $\phi(t)$ to $C$. This path has to include the edge $e_C$, and therefore the degree of $\phi$ is at least $\gamma n$. \\

	We distinguish three cases for the value of $\gamma$:
	\begin{enumerate}
		\item $\gamma = 0$. Then all vertices $v$ of $\widehat{G}$ are mapped to $\phi(t)$ under $\phi$, and therefore the degree of $\phi$ is bounded below by the number of vertices in $\widehat{G}$, which is equal to $2n + 1$. 
		\item $\gamma \geq 2$. As mentioned before, a lower bound for the degree of $\phi$ is given by $\gamma n \geq 2n$. 
		\item $\gamma = 1$. We look more closely at the size of the preimage of $e_C$. The edges attached to $t$ add at least $n$ to this size. Furthermore, every other vertex of $\widehat{G}'$ that is mapped to $\phi(t)$ increases this number by at least one, as the index of this vertex in the direction of $e_C$ is at least one. We therefore deduce a lower bound on the number of vertices other than $t$ of $\widehat{G}'$ that are mapped to $\phi(t)$: 
		\begin{itemize}
			\item Every $v \in \widehat{G} \backslash t$ such that $\phi(v) = \phi(t)$ increases this number by at least one.
			\item Every edge $uv$ of $\widehat{G} \backslash t$ such that $\phi(u) \neq \phi(t) \neq \phi(v)$ (we also have $\phi(u) \neq \phi(v)$ as $\gamma = 1$) has to be subdivided, and the unique path from $\phi(u)$ to $\phi(v)$ goes through $\phi(t)$, hence this also adds another vertex in $\widehat{G}'$ that is mapped to $\phi(t)$.
		\end{itemize}
		This shows that the degree of $\phi$ is bounded below by 
		\[
		n + |\{v \in V(\widehat{G}\backslash t) \mid \phi(v) = \phi(t)\}| + |\{uv \in E(\widehat{G}\backslash t) \mid \phi(u) \neq \phi(t) \neq \phi(v)\}|.
		\]
		By applying Lemma~\ref{lem:vertexcover} found below with $H = \widehat{G} \backslash t$ and $S = \{v \in V(\widehat{G} \backslash t) \mid \phi(v) = \phi(t)\}$, the degree of $\phi$ is bounded below by
		\[
		n + \tau(\widehat{G} \backslash t) = n + \tau(G) + 1.
		\]
	\end{enumerate} 
	
	Certainly $\tau(G) \leq n - 1$, hence $n + \tau(G) + 1 \leq 2n$. Because the three cases are exhaustive, we have 
	\[
	\textup{sgon}(\widehat{G}) \geq \min \{2n + 1, 2n, n + \tau(G) + 1\} = n + \tau(G) + 1.
	\]
	As we assumed that $\textup{sgon}(\widehat{G}) \leq n + k + 1$, we find $n + \tau(G) + 1 \leq n + k + 1$, and thus $\tau(G) \leq k$.
\end{proof}

To complete the proof of Theorem~\ref{lem:sgon-vertexcover} we conclude with a simple lemma stating that minimising a quantity of the type found in the proof of the theorem is realised by the minimal vertex cover:

\begin{lemma}\label{lem:vertexcover}
	Let $H$ be a graph. Then 
	\[
	\min_{S \subseteq V(H)} \big( |S| + |\{uv \in E(H) \mid u,v \notin S\}| \big) = \tau(H),
	\]
	where $\tau(H)$ denotes the size of a minimal vertex cover of $H$.
\end{lemma}

\begin{proof}
	If we let $S$ be any minimal vertex cover, then $|S| = \tau(H)$ and $|\{uv \in E(H) \mid u,v \notin S\}| = 0$, hence the left hand side is at most the right hand side. To argue the other inequality, assume $|S| = \tau(H) - a$ for some $a > 0$ and $|\{uv \in E(H) \mid u,v \notin S\}| < a$. For every edge in $\{uv \in E(H) \mid u,v \notin S\}$ choose one of the endpoints and add it to $S$, obtaining a new set $S'$. By construction $\{uv \in E(H) \mid u,v \notin S'\} = \emptyset$ and thus $S'$ is a vertex cover. However $|S'| = |S| + |\{uv \in E(H) \mid u,v \notin S\}| < \tau(H)$, contradicting the minimality of $\tau(H)$. Thus the desired equality holds.
\end{proof}

\begin{theorem}\label{thm:sgonnphard}
	The \textsc{Stable Gonality} problem is NP-hard. 
\end{theorem}

\begin{proof}
	Let $(G,k)$ be an instance of the \textsc{Vertex Cover} problem and let $\widehat{G}$ be the graph constructed above. By Lemma \ref{lem:sgon-vertexcover} it follows that the \textsc{Stable Gonality} problem is NP-hard. 
\end{proof}

\begin{figure}
	\centering
	\begin{subfigure}{\textwidth}
		\centering
		\begin{tikzpicture}
		\node[vertex, label=below:$u$] (u) at (0,0) {};
		\node[vertex, label=below:$v$] (v) at (2,0) {};
		\node[vertex, label=above:$v_1$] (w) at (1,1.7) {};
		\node[vertex, label=left:$t$] (t) at (-1,2) {};
		\node[vertex, label=below:$c$] (c) at (-1.5,0) {};
		\node[vertex, label=left:$v_2$] (a) at (-3,0.75) {};
		\node[vertex, label=left:$v_3$] (b) at (-3,-0.75) {};
		\draw[edge] (u) -- (v) -- (w) -- (u); 
		\draw[edge] (u) -- (c);
		\draw[edge] (a) -- (c) -- (b); 
		\draw[edge] (t) -- (a) (t) -- (b) (t) -- (c);
		\draw[edge] (t) -- (u) (t) -- (v) (t) -- (w);
		\draw[edge] (t) to[relative, in=170, out=10] (a);
		\draw[edge] (a) to[relative, in=170, out=10] (t);
		\draw[edge] (t) to[relative, in=170, out=10] (b);
		\draw[edge] (b) to[relative, in=170, out=10] (t);
		\draw[edge] (t) to[relative, in=170, out=10] (c);
		\draw[edge] (c) to[relative, in=170, out=10] (t);
		\draw[edge] (t) to[relative, in=170, out=10] (u);
		\draw[edge] (u) to[relative, in=170, out=10] (t);
		\draw[edge] (t) to[relative, in=170, out=10] (v);
		\draw[edge] (v) to[relative, in=170, out=10] (t);
		\draw[edge] (t) to[relative, in=170, out=10] (w);
		\draw[edge] (w) to[relative, in=170, out=10] (t);

		\draw[->] (3.5,.75) -- (4.5,.75);
		
		\node[vertex, label=below:$t'$] (t') at (7,.6){};
		\node[vertex, label=right:$v_1'$] (x1) at (7,1.7){};
		\node[vertex, label=above:$v_2'$] (x2) at (8,0){};
		\node[vertex, label=above:$v_3'$] (x3) at (6,0){};
		\draw[edge] (t') -- (x1);
		\draw[edge] (t') -- (x2);
		\draw[edge] (t') -- (x3);
		
		\end{tikzpicture}
		\caption{\footnotesize The map $\phi\colon \widehat{G} \to T$.}
		\label{fig:map}
	\end{subfigure}
	~ 
	
	\begin{subfigure}{\textwidth}
		\centering
		\begin{tikzpicture}
		\node[vertex, label=below:$u$] (u) at (0,0) {};
		\node[vertex, label=below:$v$] (v) at (2,0) {};
		\node[vertex] (w) at (1,1.7) {};
		\node[vertex, label=left:$t$] (t) at (-1,2) {};
		\node[vertex, label=below:$c$] (c) at (-1.5,0) {};
		\node[vertex] (a) at (-3,0.75) {};
		\node[vertex] (b) at (-3,-0.75) {};
		\draw[edge] (u) -- (v) node [midway, added] {} ;
		\draw (v) -- (w) -- (u); 
		\draw[edge] (u) -- node [midway, added] {} (c);
		\draw[edge] (a) -- (c) -- (b); 
		\draw[edge] (t) -- (a);
		\draw[edge] (t) -- (b);
		\draw[edge] (t) -- node [midway, added] {} (c);
		\draw[edge] (t) -- node [midway, added] {} (u);
		\draw[edge] (t) -- node [pos=.4, added] {} (v);
		\draw[edge] (t) -- (w);
		\draw[edge] (t) to[relative, in=170, out=10] (a);
		\draw[edge] (a) to[relative, in=170, out=10] (t);
		\draw[edge] (t) to[relative, in=170, out=10] (b);
		\draw[edge] (b) to[relative, in=170, out=10] (t);
		\draw[edge] (t) to[relative, in=170, out=10] node [midway, added] {} (c) ;
		\draw[edge] (c) to[relative, in=170, out=10] node [midway, added] {} (t) ;
		\draw[edge] (t) to[relative, in=170, out=10]  node [midway, added] {} (u);
		\draw[edge] (u) to[relative, in=170, out=10]  node [midway, added] {} (t);
		\draw[edge] (t) to[relative, in=170, out=10]  node [pos=.4, added] {} (v);
		\draw[edge] (v) to[relative, in=170, out=10]  node [pos=.6, added] {} (t);
		\draw[edge] (t) to[relative, in=170, out=10] (w);
		\draw[edge] (w) to[relative, in=170, out=10] (t);

		\draw[->] (3.5,.75) -- (4.5,.75);
		
		\node[vertex] (t') at (7,.6) {};
		\begin{scope}[xshift=7cm, yshift=.6cm]
		\foreach \x in {0,1,...,13}
		{
			\node[vertex] (x) at (25.7*\x:1.5) {};
			\draw[edge] (t') -- (x);
		}
		\end{scope}
		\end{tikzpicture}
		\caption{\footnotesize After step 1.}
		\label{fig:step1}
	\end{subfigure}
	~ 
	
	\begin{subfigure}{\textwidth}
		\centering
		\begin{tikzpicture}
		\node[vertex] (u) at (0,0) {};
		\node[vertex] (v) at (2,0) {};
		\node[vertex] (w) at (1,1.7) {};
		\node[vertex] (t) at (-1,2) {};
		\node[vertex] (c) at (-1.5,0) {};
		\node[vertex] (a) at (-3,0.75) {};
		\node[vertex] (b) at (-3,-0.75) {};
		\draw[edge] (u) -- (v) node [midway, added] {} ;
		\draw (v) -- (w) -- (u); 
		\draw[edge] (u) -- node [midway, added] {} (c);
		\draw[edge] (a) -- (c) -- (b); 
		\draw[edge] (t) -- (a);
		\draw[edge] (t) -- (b);
		\draw[edge] (t) -- node [midway, added] {} (c);
		\draw[edge] (t) -- node [midway, added] {} (u);
		\draw[edge] (t) -- node [pos=.4, added] {} (v);
		\draw[edge] (t) -- (w);
		\draw[edge] (t) to[relative, in=170, out=10] (a);
		\draw[edge] (a) to[relative, in=170, out=10] (t);
		\draw[edge] (t) to[relative, in=170, out=10] (b);
		\draw[edge] (b) to[relative, in=170, out=10] (t);
		\draw[edge] (t) to[relative, in=170, out=10] node [midway, added] {} (c) ;
		\draw[edge] (c) to[relative, in=170, out=10] node [midway, added] {} (t) ;
		\draw[edge] (t) to[relative, in=170, out=10]  node [midway, added] {} (u);
		\draw[edge] (u) to[relative, in=170, out=10]  node [midway, added] {} (t);
		\draw[edge] (t) to[relative, in=170, out=10]  node [pos=.4, added] {} (v);
		\draw[edge] (v) to[relative, in=170, out=10]  node [pos=.6, added] {} (t);
		\draw[edge] (t) to[relative, in=170, out=10] (w);
		\draw[edge] (w) to[relative, in=170, out=10] (t);
		\begin{scope}[xshift=-1.5cm]
		\foreach \x in {1,2,...,8}
		{
			\node[added] (x) at (210+16.6*\x:.5) {};
			\draw[edge] (c) -- (x);
		}
		\end{scope}
		\foreach \x in {1,2,...,8}
		{
			\node[added] (x) at (180+20*\x:.5) {};
			\draw[edge] (u) -- (x);
		}
		\begin{scope}[xshift=2cm]
		\foreach \x in {1,2,...,9}
		{
			\node[added] (x) at (180+30*\x:.5) {};
			\draw[edge] (v) -- (x);
		}
		\end{scope}
		\begin{scope}[xshift=-1cm, yshift=2cm]
		\foreach \x in {1,2}
		{
			\node[added] (x) at (60*\x:.5) {};
			\draw[edge] (t) -- (x);
		}
		\end{scope}

		\draw[->] (3.5,.75) -- (4.5,.75);

		\node[vertex] (t') at (7,.6) {};
		\begin{scope}[xshift=7cm, yshift=.6cm]
		\foreach \x in {0,1,...,13}
		{
			\node[vertex] (x) at (25.7*\x:1.5) {};
			\draw[edge] (t') -- (x);
		}
		\end{scope}
		\end{tikzpicture}
		\caption{\footnotesize After step 2.}
		\label{fig:step2}
	\end{subfigure}
	\caption{The map $\phi\colon \widehat{G} \to T$ and the step-by-step extension to a morphism where we take $\{u,v\}$ as  the vertex cover of the original triangle.}\label{fig:morphism}
\end{figure}
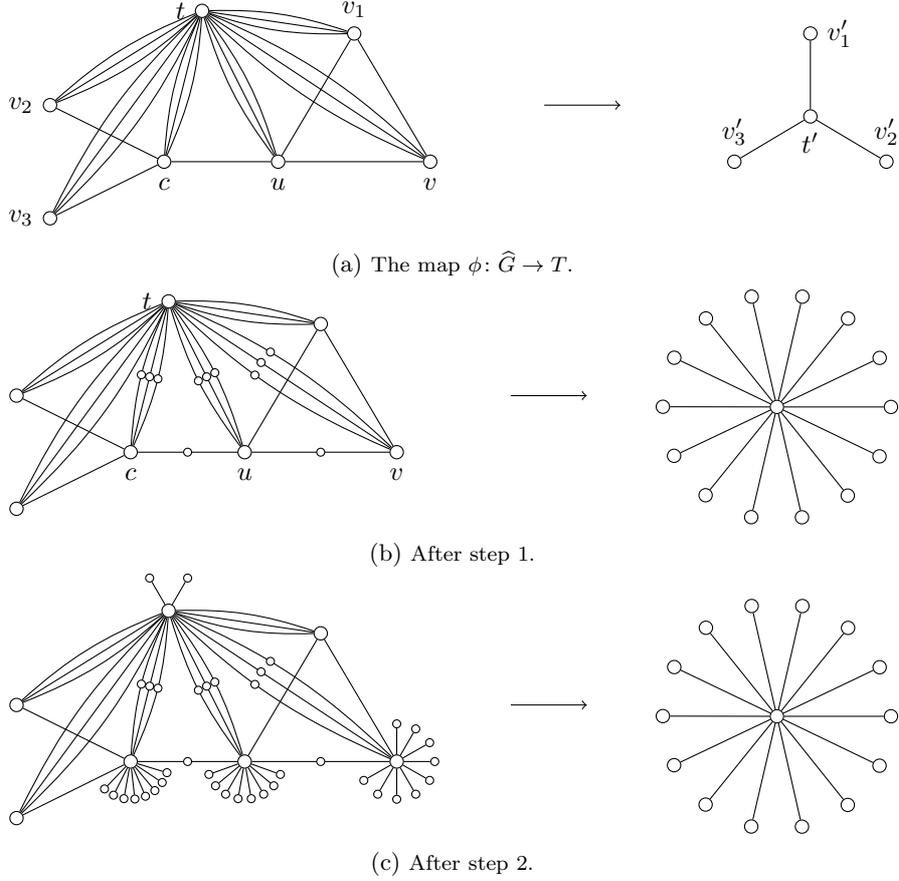

\subsection{APX-hardness}

The \textsc{Vertex Cover} problem is hard for subcubic graphs \cite{APX}. So we will use the NP-hardness reduction above to show APX-hardness of the \textsc{Stable Gonality} problem. Notice that we consider the optimisation variant of the \textsc{Stable Gonality} problem here, i.e., we ask for a finite harmonic morphism from a refinement to a tree with minimum degree.

\begin{lemma}\label{lem:sgon-apx-bound} Let $G = (V,E)$ be a subcubic graph with $|V| \geq 8$ and let $\widehat{G}$ be as constructed above.
	Let $\widehat{G}'$ be a refinement of $\widehat{G}$, $T$ a tree and $\phi$ a finite harmonic morphism of degree at most $(1+\epsilon) \sgon(\widehat{G})$. Then we can find, in polynomial time, a vertex cover $S$ of $G$ of size at most $(1+32\epsilon)\tau(G)$.
\end{lemma}
\begin{proof} 
	Define $\gamma = \max\{\big| \phi^{-1}(C) \cap \widehat{G} \big| \mid C \text{ component of } T\backslash\{\phi(t)\}\}$ as in the proof of Lemma \ref{lem:sgon-vertexcover}. We distinguish cases for the value of $\gamma$. 
	\begin{enumerate}
		\item $\gamma = 1$. 
		Construct a vertex cover of $G$ as follows. Define $S = \{v\in V \mid \phi(v) = \phi(t)\}$. For every edge in $E[V\backslash S]$, add one of its endpoints. Write $S'$ for this set. 
		
		Notice that $S'$ is a vertex cover with size at most $|S| + |E[V\backslash S]|$. From the proof of Lemma \ref{lem:sgon-vertexcover}, it follows that
		\begin{align*}
		|S'| &\leq |S| + |E[V\setminus S]| \\ 
		&\leq \deg(\phi) - n-1 \\
		&\leq (1+\epsilon)\sgon(\widehat{G}) - n -1 \\
		&= (1+\epsilon)(n + \tau(G) + 1) -n-1 \\
		&\leq \epsilon n  + (1+\epsilon)\tau(G) + \epsilon.
		\end{align*}
		Since $G$ is subcubic, it holds that $\tau(G) \geq \frac{1}{3}n$, hence
		\begin{align*}
		|S'| &\leq \epsilon 3\tau(G) + (1+\epsilon)\tau(G) + \epsilon\\
		&\leq (1+ 5\epsilon)\tau(G). 
		\end{align*}
		\item $\gamma = 0$ or $\gamma \geq 2$. In the proof of Lemma \ref{lem:sgon-vertexcover}, we have seen that $\deg(\phi) \geq 2n$. Thus
		\begin{align*}
		2n &\leq (1+\epsilon)\sgon(\widehat{G}) \\
		&\leq (1+\epsilon)(n+\tau(G) + 1).
		\end{align*}
		We derive that 
		\begin{align*}
		\epsilon &\geq \frac{n-\tau(G) - 1}{n+\tau(G)+1} \\
		&\geq \frac{n-\tau(G) - 1}{2n}. \\ 
		\end{align*}
		Since $G$ is subcubic, it holds that $\tau(G) \leq \frac{3}{4}n$. Because of this and because we assumed that $n\geq 8$, it follows that
		\begin{align*}
		\epsilon &\geq \frac{\frac{1}{4}n - 1}{2n} \\ 
		&\geq \frac{\frac{1}{8}n}{2n} = \frac{1}{16}n
		\end{align*}
		The set $V$ is clearly a vertex cover and has size $|V| \leq (1+32\epsilon)\tau(G)$, since $\tau(G) \geq \frac{1}{3}n$ for subcubic graphs. 
	\end{enumerate} 
In both cases we found a vertex cover of $G$ of size at most $(1+32\epsilon)\tau(G)$. Notice that these vertex covers can be computed in polynomial time. 
\end{proof}

\begin{theorem}
	The \textsc{Stable Gonality} problem is APX-hard. 
\end{theorem}
\begin{proof}
	In \cite{APX} it was shown that the \textsc{Vertex Cover} problem is APX-hard, even on subcubic graphs. It follows that the \textsc{Vertex Cover} problem is APX-hard on subcubic graphs with at least 8 vertices as well. From this and from Lemma \ref{lem:sgon-apx-bound} the result follows. 
\end{proof}

\section{Divisorial and stable gonality are unbounded in one another} \label{sec:untied}

In the introduction we briefly mentioned some inequalities regarding different graph invariants; a summary of all relations can be found in Figure~\ref{fig:relation2}. 
In particular, it is known that treewidth is a lower bound for all notions of gonality \cite{GonTW}, and by definition stable divisorial gonality is a lower bound for divisorial gonality.  
It is mentioned in \cite{CornelissenKatoKool} that it is a lower bound for the stable gonality as well. The idea is as follows: let $\phi\colon G' \to T$ be a finite harmonic morphism of degree $k$ from a refinement $G'$ to a tree $T$. It is possible to refine $G'$ to a graph $G''$ such that the pre-image divisor, that assigns $m_\phi(v)$ chips to every vertex in $\phi^{-1}(t)$ for some vertex $t$ of $T$, has rank at least $1$. 
Furthermore, through methods of algebraic geometry the upper bound $\frac{b_1 + 3}{2}$ has been established for stable gonality, where $b_1(G)$ is the first Betti number and equals $|E(G)| - |V(G)| + 1$ \cite{CornelissenKatoKool}. It has been conjectured that this also holds for divisorial gonality \cite{Baker2008}. 

On the other hand, $b_1$ cannot be bounded above by a function of any notion of gonality alone; consider the banana graph consisting of $2$ vertices with $m$ parallel edges between them. This graph has divisorial and stable gonality $2$ (obtained by subdividing every edge once), whilst the value of $b_1 = m - 1$ can be arbitrarily high.

The notions of gonality cannot be bounded in terms  of treewidth either. In \cite{Hendrey16}, it is shown that fan-graphs (a path with a universial vertex) have arbitrarily high divisorial gonality, but treewidth 2. This argument works for stable divisorial gonality as well. 

This section is devoted to showing that stable and divisorial gonality are \emph{unbounded} in terms of one another. We construct two families of graphs, in one of which the stable gonality is bounded but the divisorial gonality unbounded, and vice versa in the other. These results function as a justification for the different approaches taken in proving the NP-hardness of stable gonality and divisorial gonality: knowledge of one of the two invariants provides little knowledge of the other. 

Note that it follows that both the stable and divisorial gonality are unbounded in the treewidth and stable divisorial gonality, and that $b_1$ is unbounded in the divisorial gonality.

\subsection{Divisorial gonality is unbounded in stable gonality}

\begin{figure} [h]
	\centering
	\begin{subfigure}{\linewidth}
		\centering
	\begin{tikzpicture}
	\node[vertex] (v1) at (0,0) {};
	\node[vertex] (v2) at (1,0) {};
	\node[vertex] (v3) at (2,0) {};
	\node[vertex] (v4) at (3,0) {};
	\node[vertex] (v5) at (4,0) {};
	\node[vertex] (v6) at (5,0) {};
	\node[vertex] (v7) at (6,0) {};
	\node[vertex] (u1) at (0,1) {};
	\node[vertex] (u2) at (1,1) {};
	\node[vertex] (u3) at (2,1) {};
	\node[vertex] (w1) at (2,-1) {};
	\node[vertex] (w2) at (3,-1) {};
	\node[vertex] (w3) at (4,-1) {};
	\node[vertex] (u4) at (4,1) {};
	\node[vertex] (u5) at (5,1) {};
	\node[vertex] (u6) at (6,1) {};
	\draw[edge] (v1) -- (v2) -- (v3);
	\draw[edge] (v3) -- (v4) -- (v5);
	\draw[edge] (v5) -- (v6) -- (v7);
	\draw[edge] (v1) -- (u1) -- (u2) -- (u3) -- (v3);
	\draw[edge] (v3) -- (w1) -- (w2) -- (w3) -- (v5);
	\draw[edge] (v5) -- (u4) -- (u5) -- (u6) -- (v7);
	\end{tikzpicture}
\end{subfigure}

\vspace{\baselineskip}
	\begin{subfigure}{\linewidth}
		\centering
	\begin{tikzpicture}[scale=.8]
	\node[vertex] (v1) at (0,0) {};
	\node[vertex] (v2) at (1,0) {};
	\node[vertex] (v3) at (2,0) {};
	\node[vertex] (v4) at (3,0) {};
	\node[vertex] (v5) at (4,0) {};
	\node[vertex] (v6) at (5,0) {};
	\node[vertex] (v7) at (6,0) {};
	\node[vertex] (v8) at (7,0) {};
	\node[vertex] (v9) at (8,0) {};
	\node[vertex] (v10) at (9,0) {};
	\node[vertex] (v11) at (10,0) {};
	\node[vertex] (v12) at (11,0) {};
	\node[vertex] (v13) at (12,0) {};
	\node[vertex] (v14) at (13,0) {};
	\node[vertex] (v15) at (14,0) {};
	\node[vertex] (v16) at (15,0) {};
	\node[vertex] (u1) at (0,1) {};
	\node[vertex] (u2) at (1,1) {};
	\node[vertex] (u3) at (2,1) {};
	\node[vertex] (u4) at (3,1) {};
	\node[vertex] (w1) at (3,-1) {};
	\node[vertex] (w2) at (4,-1) {};
	\node[vertex] (w3) at (5,-1) {};
	\node[vertex] (w4) at (6,-1) {};
	\node[vertex] (u5) at (6,1) {};
	\node[vertex] (u6) at (7,1) {};
	\node[vertex] (u7) at (8,1) {};
	\node[vertex] (u8) at (9,1) {};
	\node[vertex] (w5) at (9,-1) {};
	\node[vertex] (w6) at (10,-1) {};
	\node[vertex] (w7) at (11,-1) {};
	\node[vertex] (w8) at (12,-1) {};
	\node[vertex] (u9) at (12,1) {};
	\node[vertex] (u10) at (13,1) {};
	\node[vertex] (u11) at (14,1) {};
	\node[vertex] (u12) at (15,1) {};
	\draw[edge] (v1) -- (v2) -- (v3) -- (v4);
	\draw[edge] (v4) -- (v5) -- (v6) -- (v7);
	\draw[edge] (v7) -- (v8) -- (v9) -- (v10);
	\draw[edge] (v10) -- (v11) -- (v12) -- (v13) -- (v14) -- (v15) -- (v16);
	\draw[edge] (v1) -- (u1) -- (u2) -- (u3) -- (u4) -- (v4);
	\draw[edge] (v4) -- (w1) -- (w2) -- (w3) -- (w4) -- (v7);
	\draw[edge] (v7) -- (u5) -- (u6) -- (u7) -- (u8) -- (v10);
	\draw[edge] (v10) -- (w5) -- (w6) -- (w7) -- (w8) -- (v13);
	\draw[edge] (v13) -- (u9) -- (u10) -- (u11) -- (u12) -- (v16);
	\end{tikzpicture}
\end{subfigure}
	\caption{Graphs $G_k$ consisting of $2k-3$ cycles of length $2k$, depicted for $k=3$ and $k = 4$. These graphs have divisorial gonality $k$ and stable gonality $2$. }\label{fig:dgon-onbegrend}
\end{figure}
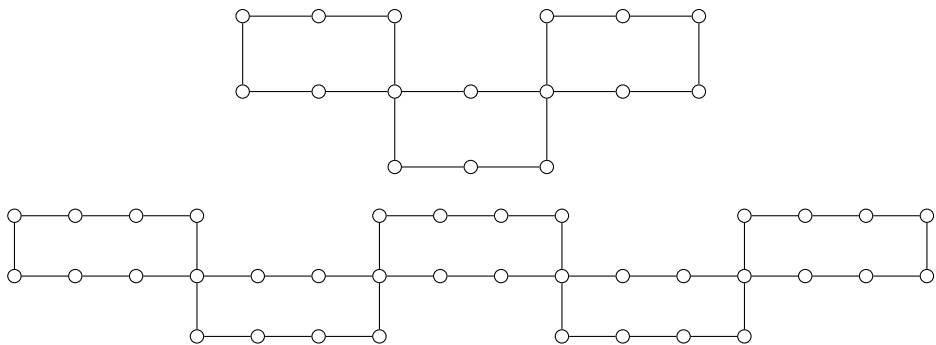
In the proof of Corollary 17 in the arXiv-version of \cite{Hendrey16} the following family of graphs is considered: let $G_k$ be the graph consisting of a sequence of $2k-3$ cycles of length $2k$, concatenated such that every cycle has an arc of length $k-1$ and an arc of length $k+1$. See Figure \ref{fig:dgon-onbegrend} for an illustration of $G_3$ and $G_4$. 
It is shown there that these graphs have divisorial gonality $k$. However, all $G_k$ have stable gonality $2$: subdivide every arc of length $k-1$ two times and map the obtained graph to a path of $(2k-3)(k+1) + 1$ vertices. This yields a finite harmonic morphism of degree 2, which shows that the divisorial gonality of a graph is not bounded by any function of the stable gonality.

\subsection{Stable gonality is unbounded in divisorial gonality}
Ye Luo \cite[Example 5.13]{Amini2} gave an example of a graph with divisorial gonality $3$ and stable gonality $4$, see Figure \ref{fig:graafLuo}, showing that graphs exist where the stable gonality is strictly larger than the divisorial gonality. We extend this result by constructing a family of graphs all of which have divisorial gonality at most $3$, but the stable gonality is unbounded.
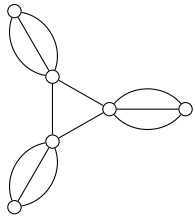
\begin{figure}
	\centering
	\begin{tikzpicture}
		\graafr{0}{0}
	\end{tikzpicture}
	\caption{A graph with stable gonality $4$ and divisorial gonality $3$.} \label{fig:graafLuo}
\end{figure}

\begin{construction} 
	Let $G_0$ be the triangle graph. For $k\geq 2$, we construct $G_k$ as follows. Take a triangle graph and three copies of $G_{k-1}$. For every copy of $G_{k-1}$, pick a vertex in this copy. Identify each of those vertices with one of the vertices of the triangle graph. We call this triangle graph the \emph{central triangle $C$} in $G_k$. 
\end{construction}
\begin{figure}
	\begin{tikzpicture}[scale = 0.62]
		\makeg{0}{0}{90}{1}
		\makegtwo{4.3}{0}{0}{1}
		\makegthree{13}{0}{0}
	\end{tikzpicture}
	\caption{Examples of graphs $G_1, G_2$, and $G_3$. }\label{fig:groteGraaf}
\end{figure}
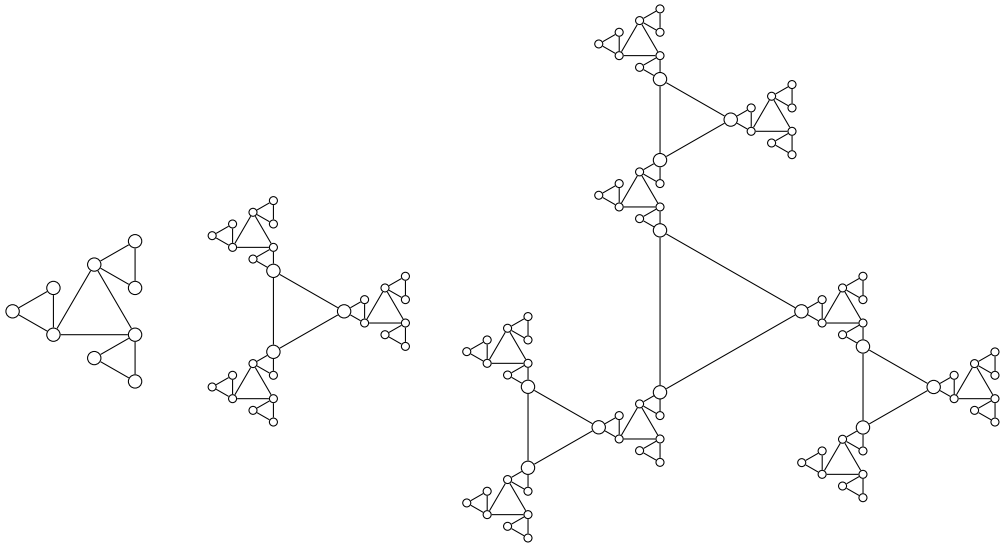

\begin{remark}
	Notice that the graphs $G_k$ constructed above are not unique. See Figure \ref{fig:groteGraaf} for an example. In the rest of this section, we fix some sequence $G_k$ that can be constructed in this way. 
\end{remark}

We will prove that the divisorial gonality of $G_k$ is 3 for all $k \geq 1$, and that the stable gonality of $G_k$ is at least $k+2$. For this, we first show some lemmas. The first lemma states that, given a finite harmonic morphism $\phi\colon G' \to T$, a path in $T$ can be lifted to a path in $G'$. This is due to harmonicity of the morphism $\phi$, and this will be the only consequence of harmonicity that we use in this section.

\begin{lemma} \label{lem:path}
	Let $G$ be a graph and $\phi\colon G' \to T$ a finite harmonic morphism from a refinement of $G$ to a tree $T$. Let $P$ be a $t_1-t_2$-path in $T$. Let $v\in V(G)$ be a vertex such that $\phi(v) = t_1$. Then there exists a vertex $u \in V(G)$ and a $v-u$-path $Q$ such that $\phi(Q)$ is a path and equals $P$.
\end{lemma}
\begin{proof}
	We show this by induction on the length of the path $\phi(v) - t$. If this length equals zero, we are done. Suppose that for every vertex with distance $k$ to $\phi(v)$, we can lift this path, i.e.\ there is a $v-u$-path $Q$ that is mapped to $P$. Let $t$ be a vertex with distance $k+1$ to $\phi(v)$. Let $t'$ be the neighbour of $t$ that is closer to $\phi(v)$. By the induction hypothesis, there is a path $v, u_1, \ldots, u_k$ such that $\phi(v), \phi(u_1), \ldots, \phi(u_k)$ equals the unique path from $\phi(v)$ to $t'$. Since $\phi$ is harmonic at $u_k$, there is a neighbour $u_{k+1}$ of $u_k$, such that $\phi(u_{k+1}) = t$. 
\end{proof}

Let $\phi\colon G_k' \to T$ be a finite harmonic morphism from a refinement of $G_k$ to a tree $T$. Let $C$ be the central triangle in $G_k$. Call the vertices of this central triangle $a_i$ for $i\in \{1,2,3\}$
Write $C'$ for the subgraph of $G_k'$ that consists of $C$ and all refinements added to the edges of $C$, see Figure \ref{fig:central-copy}. Note that the vertices added to the vertices $a_i$ are not part of $C'$. 
\begin{lemma}\label{lem:niet-uniek}
	Consider the restriction of $\phi$ to $C'$. There is at least one vertex $a_i$ such that there is another vertex in $C'$ that is mapped to $\phi(a_i)$. 
\end{lemma}
\begin{figure} 
	\centering
	\begin{tikzpicture}
		\makegtwolarge{0}{20}{0}{1}
	\end{tikzpicture}
~\qquad	\qquad	\qquad
	\begin{tikzpicture}[anchor=base, baseline=-91pt]
		\central{100}{0}{0}{1}
	\end{tikzpicture}
 \caption{Left: A refinement of the graph $G_1$. Right: The refinement $C'$ of the central triangle.}
  \label{fig:central-copy}
\end{figure}
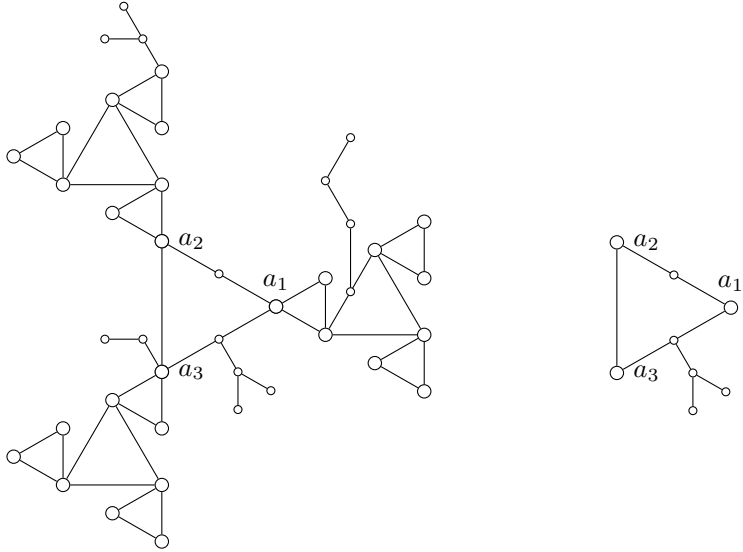
\begin{proof}
	Assume that $a_1$, $a_2$ and $a_3$ are all mapped to unique vertices by $\phi$ restricted to $C'$. Denote them by $a_1'$, $a_2'$ and $a_3'$ respectively. 
	Consider the $a_i'-a_j'$-paths in $T$. Without loss of generality, the distance $d(a_2', a_3')$ is maximal. Notice that this distance is at least 2, so there is an internal vertex $b'$ on the $a_2'-a_3'$-path. Now consider the $b'-a_1'$-path $P$. Notice that $P$ does not contain both $a_2'$ and $a_3'$, because the distance $d(a_2', a_3')$ is maximal. 
	
	Let $b$ be a subdivision of the edge $a_2a_3$ of $G$ that is mapped to $b'$. By Lemma \ref{lem:path}, there is a $b-u$-path $Q$ in $G_k$ that is mapped to $P$. Notice that $Q$ does not contain $a_2$ and $a_3$, because $a_2'$ and $a_3'$ are not contained in $P$. It follows that $Q$ lies in $C'$. Since every path from $b$ to $a_1$ does contain $a_2$ or $a_3$, it follows that $u\neq a_1$. Hence, there is another vertex in $C'$ that is mapped to $a_1'$. 
\end{proof}

\begin{lemma} \label{lem:bound-sgon}
	The stable gonality of $G_k$ is at least $k+2$ for every $k\geq 0$. 
\end{lemma}
\begin{proof}
	We prove this by induction. It is known that $G_0$ has stable gonality $2$, so the statement holds for $k=0$. 
	
	Suppose that $\sgon(G_{k-1}) \geq k+1$. Consider $G_{k}$. 
	Let $\phi\colon G_k' \to T$ be a finite harmonic morphism from a refinement of $G_k$ to a tree $T$. By Lemma \ref{lem:niet-uniek}, we know that there is a vertex $a_i$, such that there is another vertex $w$ in $C'$ that is mapped to $\phi(a_i)$. 
	
	Consider the copy of $G_{k-1}$ in $G_k$ that shares vertex $a_i$, call this copy $H$. Write $H'$ for the refinement of $H$ in $G_k'$. Notice that $\phi|_H\colon H\to \phi(H)$ is a finite morphism, and is harmonic for all vertices except $a_i$. We will refine $H$ to extend the restriction of $\phi$ to $H$ to a finite harmonic morphism $\psi$. Consider the tree $T' = \phi(H)$. Let $e \in E_{\phi(a_i)}$ be the edge such that $m_{\phi,e}(a_i)$ is maximal. For every edge $e' = \phi(a_i)u$ with $m_{\phi,e'}(a_i) < m_{\phi,e}(a_i)$, we do the following. Let $T'_{\phi(a_i)}(u)$ be the subtree of $T'$ that consists of $\phi(a_i)$ together with the component that contains $u$ after removing $\phi(a_i)$ from $T'$. We add a copy of $T'_{\phi(a_i)}(u)$ to $a_i$, and assign index $m_{\phi,e}(a_i) - m_{\phi,e'}(a_i)$ to these new edges. Set $\psi$ as the identity map for those new vertices. For all vertices $v$ of $H$, set $\psi(v) = \phi(v)$. This map $\psi$ is a finite harmonic morphism. By the induction hypothesis it follows that $\deg(\psi) \geq k+1$. 
	
	We compute the degree of $\phi$ by counting the pre-images of $\phi(a_i)$: \begin{align*}
	\deg(\phi) &= \sum_{v\in G_k', \phi(v) = \phi(a_i)} r_\phi(v) \\
	&\geq r_\phi(w) +  \sum_{v\in H', \phi(v) = \phi(a_i)} r_\phi(v) \\
	&\geq 1 +  \sum_{v\in H', \psi(v) = \psi(a_i)} r_\psi(v) \\
	&\geq 1+ \deg(\psi) \\
	&\geq k+2. 
	\end{align*}
	We conclude that $\sgon(G_k) \geq k+2$. 
\end{proof}

\begin{theorem}
	For every $k\geq 1$ there is a graph with divisorial gonality $3$ and stable gonality at least $k+2$. 
\end{theorem}
\begin{proof}
	Notice that $G_1$ has divisorial gonality 3. Moreover, the divisor with $3$ chips on $a_1$ has rank at least $1$ and for every vertex $v$, this divisor is equivalent to the divisor with $3$ chips on $v$. It follows that the divisor on $G_k$ with $3$ chips on a vertex $v$ has degree $3$ and rank at least $1$. We conclude that for every $k$ the graph $G_k$ has divisorial gonality at most $3$. 
	
	By Lemma \ref{lem:bound-sgon}, we see that the graph $G_{k}$ has stable gonality at least $k+2$. 
\end{proof}

\begin{remark}
    This construction can be carried out in a more general way than what has been done in this section. Let $G_1$ be any graph with divisorial gonality $d$ divisible by $3$ and stable gonality $s$. Suppose that a divisor on $G_1$ of degree $d$ and rank at least $1$ is equivalent to a divisor with $d$ chips on a single vertex $v$. Take a triangle graph and three copies of $G_1$. In each of the graphs $G_1$ identify the vertex $v$ with one of the vertices of the triangle graph (so that no vertex of the triangle is chosen twice). We call this graph $G_2$, and it has divisorial gonality $d$ and stable gonality at least $s+1$. Repeating this recursive process generates a family of graphs with bounded divisorial gonality and unbounded stable gonality.
\end{remark}

\section{Conclusion}
In this paper we have seen several notions of gonality for graphs. We have proven that all are NP-hard to compute. This leaves open some questions regarding the exact complexity of the notions of graph gonality. For example, it is known that the \textsc{Divisorial Gonality} problem is in XP, but it is remains open whether it is in FPT or W[1]-hard. For the \textsc{Stable Gonality} problem it is even unknown whether it is in XP. It is unknown whether any of the notions of gonality for graphs can be used as parameter for FPT-algorithms: it is interesting to see NP-hard problems that are not tractable on graphs of bounded treewidth, but are tractable on graphs of bounded gonality. 

Moreover, we have seen that all notions of graph gonality are APX-hard. 
As mentioned before, it is unknown whether any of these is in APX.

\end{document}